\documentclass{amsart}

\usepackage{latexsym,amsfonts,amssymb,mathrsfs,graphicx,subfigure,stmaryrd}
\usepackage{color,tikz}

\theoremstyle{plain}
\newtheorem{theorem}{Theorem}
\newtheorem{lemma}{Lemma}[section]

\theoremstyle{definition}

\newtheorem{example}{Example}

\numberwithin{equation}{section}

\theoremstyle{remark}
\newtheorem*{remark}{Remark}

\def\arg{\operatorname{arg}}
\def\div{\operatorname{div}}

\newcommand{\mb}[1]{\mathbb #1}
\newcommand{\mc}[1]{\mathcal #1}

\newcommand{\wt}[1]{\widetilde #1}

\newcommand{\abs}[1]{\lvert#1\rvert}
\newcommand{\labs}[1]{\left\lvert\,#1\,\right\rvert}

\newcommand{\set}[2]{\{\,#1\,\mid\,#2\,\}}
\newcommand{\Lr}[1]{\left(#1\right)}
\newcommand{\aver}[1]{\left\langle\,#1\,\right\rangle}
\newcommand{\nm}[2]{\|\,#1\,\|_{#2}}

\newcommand{\red}{\color{red}}
\newcommand{\barint}{\kern4pt \raise3.4pt\hbox{\vrule height.6pt
		width7pt} \kern-11pt \int}

\def\na{\nabla}
\def\dx{\,\mathrm{d}\,x}

\newcommand{\eps}{\varepsilon}

\def\a{a^{\;\varepsilon}}
\def\uu{u^{\,\varepsilon}}
\def\vv{v^{\,\varepsilon}}
\def\pa{\partial}
\def\hmm{\text{MOD}}

\def\Lam{\Lambda}
\def\al{\alpha}
\def\del{\delta}
\def\x{\times}

\begin{document}
\title{An Efficient Online-Offline Method for Elliptic Homogenization Problems}
\author[Y.F. Huang]{Yufang Huang}
\address{LSEC, Institute of Computational Mathematics and Scientific/Engineering Computing, AMSS, Chinese Academy of Sciences, No. 55, East Road Zhong-Guan-Cun, Beijing 100190, China\\
	and School of Mathematical Sciences, University of Chinese Academy of Sciences, Beijing 100049, China}
\curraddr{Cornell University, 425 East 61st Street, New York, 10065, USA}
 \email{huangyufang@lsec.cc.ac.cn}

\author[P.B. Ming]{Pingbing Ming}
\address{LSEC, Institute of Computational Mathematics and Scientific/Engineering Computing, AMSS, Chinese Academy of Sciences, No. 55, East Road Zhong-Guan-Cun, Beijing 100190, China\\
 and School of Mathematical Sciences, University of Chinese Academy of Sciences, Beijing 100049, China}
 \email{mpb@lsec.cc.ac.cn}
 
\author[S.Q. Song]{Siqi Song}
\address{LSEC, Institute of Computational Mathematics and Scientific/Engineering Computing, AMSS, Chinese Academy of Sciences, No. 55, East Road Zhong-Guan-Cun, Beijing 100190, China\\
	and School of Mathematical Sciences, University of Chinese Academy of Sciences, Beijing 100049, China}
\email{songsq@lsec.cc.ac.cn}

\thanks{The second author would like to thank Professor Ruo Li for his helpful discussion in the earlier stage of the present work. P. B. Ming and S. Q. Song are supported by National Natural Science Foundation of China through Grant No. 11971467 and Beijing Academy of Artificial Intelligence (BAAI).}
\date{\today}

\begin{abstract}
We present a new numerical method for solving the elliptic homogenization problem. The main idea is that the missing effective matrix is reconstructed by solving the local least-squares in an offline stage, which shall be served as the input data for the online computation. The accuracy of the proposed method are analyzed with the aid of the refined estimates of the reconstruction operator. Two dimensional and three dimensional numerical tests confirm the efficiency of the proposed method, and illustrate that this online-offline strategy may significantly reduce the cost without loss of accuracy.%
\end{abstract}

\maketitle
\section{Introduction}
We consider a prototypical elliptic boundary value problem
 \begin{equation}\label{eq:bvp}
\left\{\begin{aligned}
    -\div (\a(x) \nabla\uu(x))&=f(x),\qquad&& x\in D\subset\mb{R}^d,\\
    \uu(x)&=0, \qquad && x\in\pa D,\\
   \end {aligned}\right.
\end{equation}
where $\eps$ is a small parameter that signifies explicitly the multiscale nature of the problem. We assume that the coefficient $\a$, which is not necessarily symmetric, belongs to a set $\mc{M}(\al,\beta,D)$ that is defined by
\begin{equation}\label{eq:set}
\begin{aligned}
\mc{M}(\al,\beta,D){:}=\{\mc{B}\in[L^\infty(D)]^{d^2}|&(\mc{B}\xi,\xi)\geq \al \abs{\xi}^2,\abs{\mc{B}(x)\xi}\le\beta\abs{\xi},  \\
&\text{ for any } \xi\in \mb{R}^d \text{ and a.e. } x \in D\},
\end{aligned}
\end{equation}
where $D$ is a bounded domain in $\mb{R}^d$ and $(\cdot,\cdot)$ denotes the  inner product in $\mb{R}^d$ , while $\abs{\cdot}$ is the corresponding norm.

In the sense of H-convergence~\cite{Tartar:1977, Tartar:2010}, for every sequence $\a\in\mc{M}(\al,\beta,D)$ and $f\in H^{-1}(D)$, the sequence $\uu$ of the solution to~\eqref{eq:bvp} satisfies
\begin{equation}\label{Hconvergence}
\left\{\begin{aligned}
  \uu\rightharpoonup u_0 \qquad &\text{ weakly in } H_0^1(D), \\
 \a\na\uu\rightharpoonup \mc{A}\na u_0\qquad &\text{ weakly in } [L^2(D)]^d,
\end{aligned}\right.
\qquad\text{as\quad}\eps\to 0,
\end{equation}
where $u_0$ is the solution of the homogenization problem
\begin{equation}\label{eq:homo}
\left\{\begin {aligned}
    -\div (\mc{A}(x)\na u_0(x))&=f(x),\qquad  &&x\in D,\\
    u_0(x)&=0, \qquad && x\in\pa D,
   \end{aligned}\right.
\end{equation}
and $\mc{A}\in\mc{M}(\al,\beta,D)$. Here $H_0^1(D),L^2(D)$ and $H^{-1}(D)$ are standard Sobolev spaces~\cite{Adams:2003}.

The quantities of interest for Problem~\eqref{eq:bvp} and Problem~\eqref{eq:homo} are the homogenized solution $u_0$ over the whole domain and the solution $\uu$ at certain critical local region. The former stands for the information at the large scale, and the later mimics the information at small scale. There are lots of work devoted to efficiently compute such quantities during the last several decades; see, e.g.,~\cite{Babuska:1976, EHou:2009, E:2011}, among many others. Presently we are interested in the efficient way to compute $u_0$. 
A typical way that towards this is provided by the heterogeneous multiscale method (HMM)~\cite{EE:2003, Abdulle:2012}, and the FE$^2-$method~\cite{Michel:1999} commonly used in the engineering community that is also in the same spirit of HMM. The underlying idea of this approach is to extract $\mc{A}$ by solving the cell problems posed on the sampling points of the macoscopic solver. At each point, one needs to solve $d$ cell problems with $d$ the dimensionality. Therefore, the main computational cost comes from solving all these cell problems. The number of the cell problems grows rapidly when higher-order macroscopic solvers are employed. To reduce the cost, certain nonconventional quadrature schemes were proposed in~\cite{Du2010} when finite element method is used as the macroscopic solver. The number of the cell problems reduces to one third compared to the standard mid-point quadrature scheme when $\mb{P}_2$ Lagrange finite element method is employed as the macroscopic solver. Unfortunately, it does not seem easy to extend such idea to even higher order macroscopic solvers because the quadrature nodes tend to accumulate in the interior of the element~\cite{Stroud:1970, Schmid:1978, solin2003higher}.

In~\cite{LiMingTang:2012}, the authors presented a local least-squares reconstruction of the effective matrix using the solution of the cell problems posed on the vertices of the triangulation, which was dubbed as HMM-LS. The total number of the cell problems equals to the total number of the interior vertices of the triangulation, which is of $\mc{O}(h^{-d})$ with $h$ the mesh size of the macroscopic solver. This method achieves higher-order accuracy with almost the same cost of HMM with $\mb{P}_1$ Lagrange finite element method as the macroscopic solver~\cite{EMingZhang:2005}. A drawback of this method is that the number of the cell problems is still quite large when mesh refinement is necessary. Moreover, if the adaptive strategy is used in the macroscopic solver, then one has to solve many cell problems around the regions with mesh refinement.

In this work, we propose an offline-online method to compute $u_0$ efficiently. The main idea is to separate the microscopic solver from the macroscopic solver. In the offline stage, we firstly solve the cell problems posed on a sampling point set and obtain the effective matrix at all these points, then we reconstruct an effective matrix locally by solving the discrete least-squares. The sampling set is constructed from a triangulation of domain, which is usually coarser than the online triangulation. In the online stage, we solve the macroscopic problem with the effective matrix prepared in the offline stage.  Such decoupled strategy brings more flexibilities to reduce the number of the cell problems, we may either refine the offline triangulation mesh or increase the reconstruction order, which is guided by the a priori error estimate. The offline computation bears certain similarity with h-p finite element method~\cite{Babuska:1981, Schwab:1998}. Moreover, the offline computation is no longer linked to the macroscopic solvers, which is particularly attractive to higher order macroscopic solver and three dimensional multiscale problems. With the aid of the theoretical results proved in~\cite{Li:2019, LiMingTang:2012}, we study the accuracy and the stability of the reconstruction procedure, which is crucial to prove the optimal error estimate of the proposed method. As illustrated by the numerical tests in \S~\ref{sec:test}, the offline-online method converges with optimal order while the cost is smaller than both HMM and HMM-LS method. 

The reduced basis HMM proposed in~\cite{AbdulleBai:2012, AbdulleBai:2013} also employed the offline-online idea. The difference between reduced basis HMM and our method lies in the following points: Firstly they employed reduced basis idea in the offline stage while we construct an independent triangulation upon which the cell problems are solved. Secondly they used the empirical interpolation method~\cite{Maday:2004} while we resort to a local least-squares to reconstruct the effective matrix. Finally, a thorough analysis of the least-squares reconstruction is conducted in our work, which concerns the approximation accuracy and the stability of the reconstruction, such rigorous theoretical results put the method on a firm footing. In addition, the analysis of the least-squares reconstruction is of independent interest for other problems such as the construction of the optimal polynomial admissible meshes~\cite{Calv:2008, Piazzon:2016}, the discrete norm for polynomials~\cite{Rak:2007}, the approximation of the Fekete points~\cite{Bos:2011} and the discontinuous Galerkin method based on patch reconstruction~\cite{ Li:2019, LiYang:2020}.

The rest of the paper is organized as follows. In \S~\ref{sec:method}, we introduce the offline-online method that is based on a local discrete least-squares reconstruction. We derive the optimal error estimate in \S~\ref{sec:conv}, in particular, we prove the discrete least-squares is stable with respect to small perturbation.  In \S~\ref{sec:test}, we report numerical examples in two and three dimensions, the coefficient $\a$ may be locally periodic, almost periodic and random checker-board. To demonstrate the efficiency of the offline-online method, we compare it with HMM and HMM-LS method. In addition, we solve a problem posed on L-shape domain with nonsmooth solution. The conclusions are drawn in the last section.

Throughout this paper, we shall use the standard notations for the Sobolev spaces, norms and semi-norms, cf.,~\cite{Adams:2003},
e.g.,
\[
    \nm{v}{H^1(D)}{:}=\nm{v}{L^2(D)}+\nm{\na v}{L^2(D)},\qquad
\abs{v}_{W^{m,p}(D)}{:}=\sum_{\abs{\al}=m}\nm{\na^{\al}v}{L^p(D)}.
\]
For any measurable set $E$, we define the mean of an integrable function $g$ over $E$ as
\[
\aver{g}_E{:}=\dfrac1{\abs{E}}\int_E g(x)\dx.
\]

We shall also use the discrete $\ell_p$ norm for any $x\in\mb{R}^n$ as
\[
\nm{x}{\ell_p}{:}=\begin{cases}
\Lr{\sum_{i=1}^n\abs{x_i}^p}^{1/p}&1\le p<\infty,\\
\max_{1\le i\le n}\abs{x_i}&p=\infty.
\end{cases}
\]
Throughout the paper the generic constant $C$ may be different from line to line, while it is independent of $\eps$ and the mesh size parameters $h,H$.
\section{The Offline-online Method}\label{sec:method}
The macroscopic solver is chosen as the standard $\mb{P}_l$ Lagrange finite element method ($\mb{P}_l$ FEM)~\cite{Ciarlet:1978}. $\mb{P}_l$ is defined as the set of polynomials with degree less than $l$ for the sum of all variables. The finite element space is denoted by $V_h$ corresponding to the triangulation $\tau_h$ with mesh size $h$ that is the maximum of the element size $h_{\tau}$ for all elements $\tau\in\tau_h$, where $h_{\tau}$ is the diameter of $\tau$. We assume that all the elements $\tau$ in $\tau_h$ satisfy the shape-regular condition in the sense of Ciarlet and Raviart~\cite{Ciarlet:1978}, i.e., there exists a constant $\sigma_0$ such that $h_{\tau}/\rho_{\tau}\le\sigma_0$, where $\rho_{\tau}$ is the diameter of the smallest ball inscribed into $\tau$, and $\sigma_0$ is the so-called chunkiness parameter~\cite{brenner2007mathematical}.

The method consists of offline part and online part. In the offline part, we approximate the effective matrix $\mc{A}$ as follows.
\vskip .5cm
\noindent\textbf{Offline\;}We firstly construct a sampling triangulation $\mc{T}_H$ with mesh size $H$ over domain $D$. For simplicity, we assume that $\mc{T}_H$ consists of simplices, and  
$\mc{T}_H$ is assumed to be shape-regular with the chunkiness parameter $\sigma$. On each element $K\in\mc{T}_H$, the approximation effective matrix $A_H$ is reconstructed by solving a least-squares: for $i,j=1,\cdots,d$,
 \begin{equation}\label{eq:ls}
(A_H)_{ij}=\arg\min\limits_{p\in\mb{P}_m(S(K))}\sum\limits_{x_K\in\mc{I}(K)}\labs{
(A_H(x_K))_{ij}-p(x_K)}^2.
\end{equation}
Here $\mc{I}(K)$ is the set of all sampling points that belong to $S(K)$, where $S(K)$ is a patch of elements around $K$, which usually includes $K$. Its precise definition will be given later on. We refer to Fig.~\ref{fig:sample} for an example of such $S(K)$.

At each sampling point $x_K$, the effective matrix $A_H(x_K)$ is defined by averaging the flux arising from the cell problems:
\begin{equation}\label{eq:effective}
A_H(x_K)=\Lr{\aver{a^\eps\na v_1^\epsilon}_{I_\delta},\cdots,\aver{a^\epsilon\nabla v_d^\epsilon}_{I_\delta}},
\end{equation}
where the cell $I_\delta(x_K){:}=x_K+\delta Y$ with $Y{:}=(-1/2,1/2)^d$ and $\delta$ the cell size. Here for $i=1,\cdots,d$, $v_i^\eps$ satisfies
\begin{equation}\label{eq:cell}
\left\{\begin{aligned}
-\na\cdot(\a\na v_i^\eps) = 0& \qquad \text{in\quad}  I_\delta(x_K),\\
v_i^\eps=x_i&\qquad \text{on\quad}\pa I_\delta(x_K).
\end{aligned}\right.
\end{equation}
\vskip .5cm
\noindent\textbf{Online\;}Given $A_H$, we find $u_h\in V_h$ such that
\begin{equation}\label{eq:online}
\int_DA_H(x)\na u_h\cdot\na v\dx=\int_Df(x)v(x)\dx\quad \text{for all\quad} v\in V_h.
\end{equation}
\begin{figure}[h]
	\begin{minipage}{0.48\linewidth}
		\centering
		\includegraphics[width=7cm]{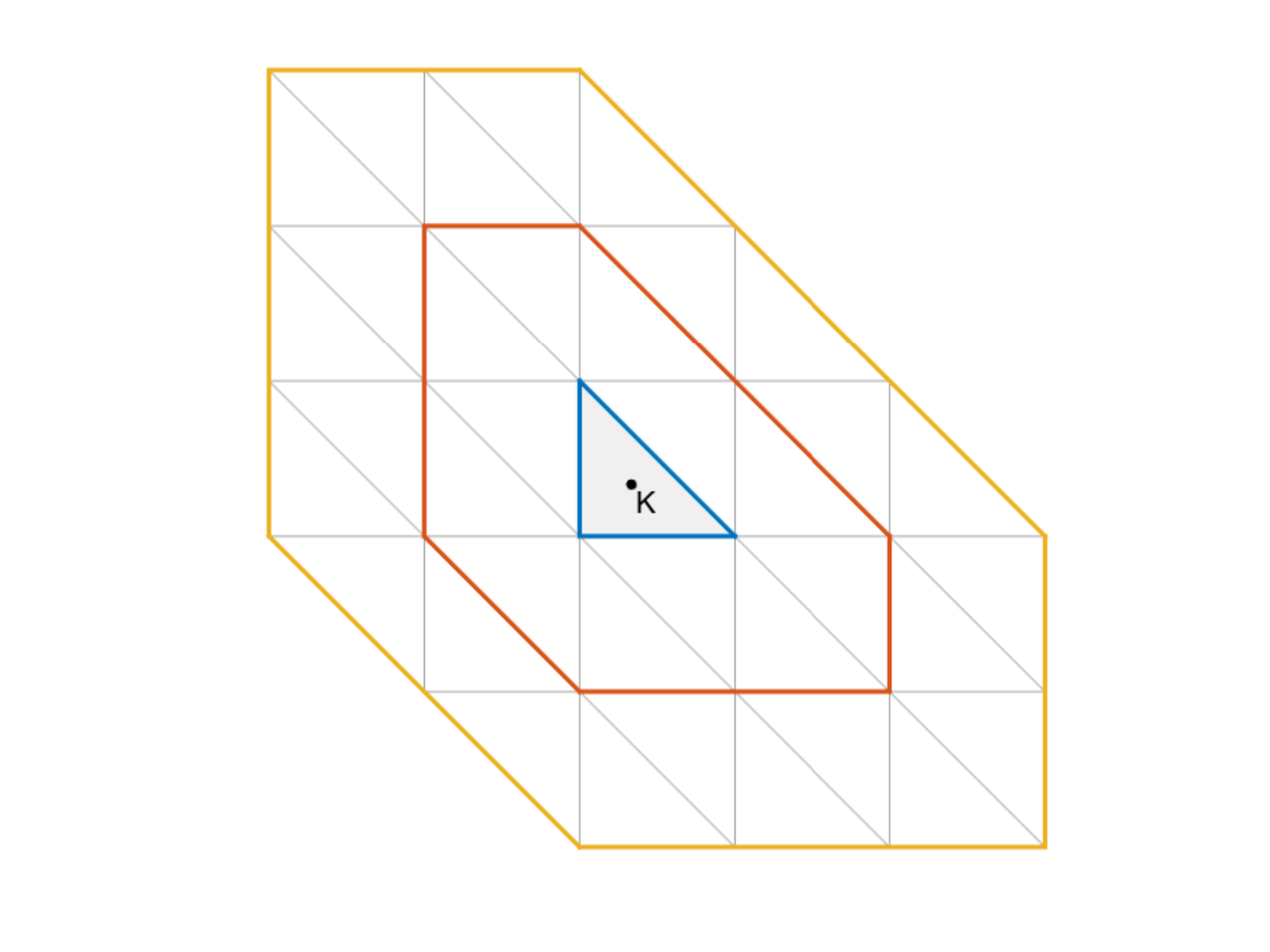}
		\center{(a). Example of $S_2(K)$ constructed by including all the Moore neighbors.}
	\end{minipage}
	\hfill
	\begin{minipage}{0.48\linewidth}
		\centering
		\includegraphics[width=6cm]{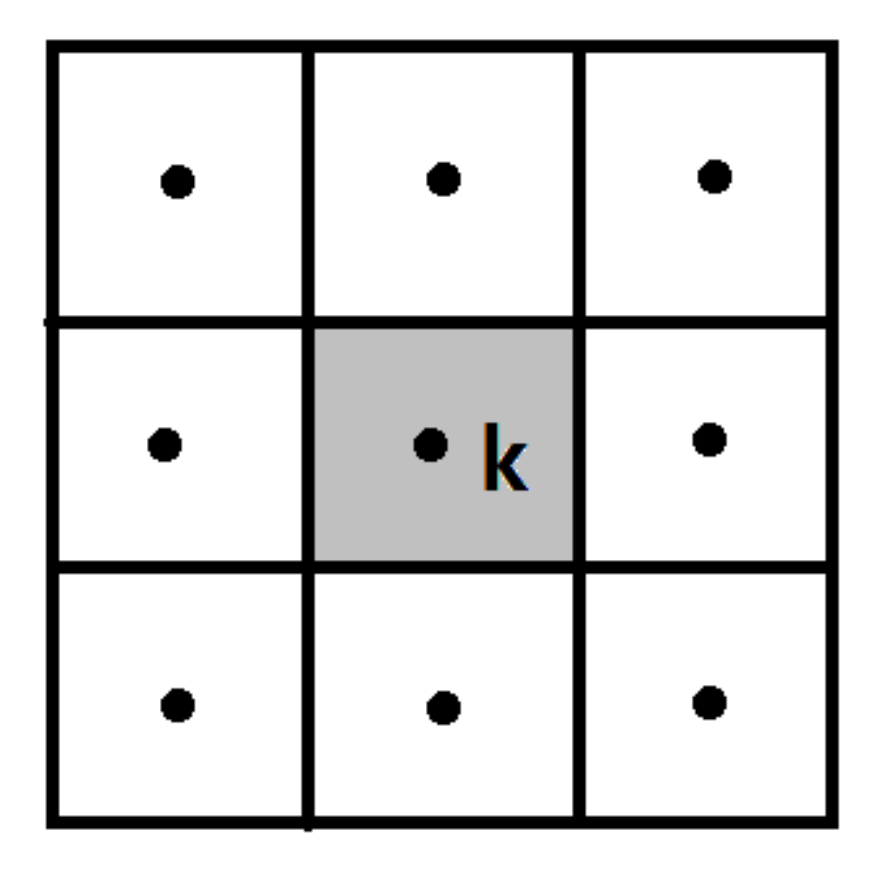}
		\center{(b). Example of $S(K)$ and the set $\mc{I}(K)$ consists of the black dots.}
	\end{minipage}
	\caption{Examples of the element patches and the sample set.}\label{fig:sample}
\end{figure}

%

\begin{remark}
The element $K\in\mc{T}_H$ may not be a simplex, which may be polygons or polytopes, the corresponding shape-regular condition and other mesh conditions may be found in~\cite{Li:2019}. Under these mesh conditions, the properties of the reconstruction are still valid.
\end{remark}

In what follows, we supplement some details in the algorithm. The first thing is the construction of the element patch $S(K)$ for any element $K\in\mc{T}_H$. We start from assigning a threshold value $N_{\text{lowest}}$ that is used to control the size of $S(K)$. There are several different ways to find $S(K)$. One way is to define $S(K)$ in an recursive way as in~\cite{LiMingTang:2012}: For any $t\in\mb{N}$, we let
\begin{equation}\label{eq:elepatch}
S_0(K){:}= K,\quad S_t(K)=\set{K\in\mc{T}_H}{\overline{K}\cap\overline{S_{t-1}(K)}\not=\emptyset}.
\end{equation}
Once $\# S_t(K)\geq N_{\text{lowest}}$, we stop the construction and let $S(K)=S_t(K)$. This means that we add the Moore neighbors~\cite{Su2001} to $S(K)$ in a recursive way. Another way is using the Von Neumann neighbor~\cite{Su2001}, i.e., we include the adjacent edge-neighboring elements into the element patch instead of the Moore neighbor. We refer to~\cite{Li:2019} and~\cite[Appendix A]{LiYang:2020} for a detailed description for such construction, while $S(K)$ in all the tests in \S~\ref{sec:test} are defined as in~\eqref{eq:elepatch}. We denote by $\mc{I}_t(K)$ the set containing all the sampling points that belong to $S_t(K)$. In all the tests in \S~\ref{sec:test}, we use the barycentric of each element $K$ as the sampling point. There are also other choices for construction $\mc{I}_t(K)$. However, the vertices are not preferred because the communication cost is quite high for three dimensional problems, though this is a good choice for two dimensional problems; cf.,~\cite{LiMingTang:2012}. In what follows, we may drop the subscripts $t$ in $S_t(K)$ and $\mc{I}_t(K)$ when there is no confusion may occur.

We remark that there are many other variants for the definitions of the cell problems and the effective matrix in the literatures; see e.g.,~\cite{Yue:2007, Gloria:2016}. The periodic cell problem will be discussed in \S~\ref{sec:test}. As illustrated by the examples in \S~\ref{sec:test}, the overhead caused by the local least-squares is small compared to the computational cost of solving all the cell problems, though the number of the local least-squares is also the same with the cell problems.
\section{Convergence of the Method}\label{sec:conv}
To study the convergence of the method, we define
\[
e(\hmm){:}=\max_{x\in D}\nm{\mc{A}(x)-A_H(x)}{F},
\]
where $\nm{B}{F}$ is the Frobenius norm of a $d-$by$-d$ matrix $B$.

Similar to~\cite[Lemma 3.1, Lemma 3.2]{LiMingTang:2012}, the following lemma gives the error estimates of the proposed method.
\begin{lemma}\label{lema:1staccuracy}
Let $u_0$ be the homogenized solution  with the homogenized matrix in $\mc{M}(\al,\beta,D)$. If $e(\hmm)<\al$, then there exists a unique solution $u_h$ satisfying~\eqref{eq:online}.

If $e(\hmm)<\nu\al/(1+\nu)$ for any $\nu>0$, then 
\begin{equation}\label{eq:h1err}
\nm{\na(u_0-u_h)}{L^2(D)}\le\dfrac{\beta}{\al}\inf_{v\in V_h}\nm{\na(u_0-v)}{L^2(D)}
+\dfrac{(1+\nu)C_p}{\al^2}\nm{f}{H^{-1}(D)}e(\hmm),
\end{equation}
 where $C_p$ appears in the discrete Poincar\'e inequality that only dependends on $D$:
\[
\nm{v}{H^1(D)}\le C_p\nm{\na v}{L^2(D)}\qquad\text{for all\quad}v\in V_h.
\]

Moreover, there exists $C$ depending only on $\al,\beta,\nu,C_p$ and $\nm{f}{H^{-1}(D)}$ such that
\[
\begin{aligned}
\nm{u_0-u_h}{L^2(D)}&\le C\Lr{\inf_{v\in V_h}\nm{\na(u_0-v)}{L^2(D)}\sup_{\nm{g}{L^2(D)}=1}\inf_{\chi\in V_h}\nm{\na(\phi_g-\chi)}{L^2(D)}+e(\hmm)},
\end{aligned}
\]
where $\phi_g\in H_0^1(D)$ is the unique solution of the problem:
\begin{equation}\label{eq:auxprob}
\int_D\mc{A}(x)\na v\cdot\na\phi_g\dx=\int_D g(x)v(x)\dx\qquad\text{for all\quad}
v\in H_0^1(D).
\end{equation}
\end{lemma}

The proof of the above lemma follows the same line of~\cite[Theorem 1.1]{EMingZhang:2005} except the explicit constants in~\eqref{eq:h1err}, we omit the proof.

To further elucidate the error structure of the method, we define the reconstruction operator for any piecewise constant function $v$ defined on $\mc{T}_H$ as follows. Let $\mc{R}_K v$ be the solution of the least-squares
\begin{equation}\label{eq:ls2}
\mc{R}_K v=\arg\min\limits_{p\in\mb{P}_m(S(K))}\sum\limits_{x_K\in\mc{I}(K)}\labs{
v(x_K)-p(x_K)}^2.
\end{equation}
We imbed $\mc{R}_K$ into a global operator as $\mc{R}|_K=\mc{R}_K$.

Given the reconstruction operator $\mc{R}$, we decompose $e(\text{MOD})$ as
\begin{equation}\label{eq:decmod}
\begin{aligned}
\mc{A}(x)-A_H(x)&=\mc{A}(x)-\mc{R}\mc{A}(x)+\mc{R}\mc{A}(x)-A_H(x)\\
&=\Lr{\mc{A}(x)-\mc{R}\mc{A}(x)}+\mc{R}\Lr{\mc{A}(x)-\overline{A}_H},
\end{aligned}
\end{equation}
where $\overline{A}_H(x)$ is a piecewise matrix posed on $\mc{T}_H$ that is defined by
\[
\overline{A}_H|_K{:}=A_H(x_K),
\]
with $A_H(x_K)$ given by~\eqref{eq:effective} and~\eqref{eq:cell}.

The first term in the right-hand side of~\eqref{eq:decmod} is the reconstruction error while the second one is the so-called estimation error. To quantify these two terms, we need  some properties of the reconstruction operator, which is by now well-understood by virtue of~\cite{LiMingTang:2012} and~\cite{Li:2019}. To state such properties, we make two assumptions on $S(K)$ and $\mc{I}(K)$.
\vskip .5cm
\noindent\textbf{Assumption A}
For every $K\in\mc{T}_H$, there exist constants $R$ and $r$ that are independent of $K$
such that $B_r\subset S(K)\subset B_R$ with $R\ge 2r$, and $S(K)$ is star-shaped with respect to $B_r$, where $B_{\rho}$ is a disk with radius $\rho$.

This assumption concerns the geometry of $S(K)$, which is crucial for the uniform boundedness of $\mc{R}$. The motivation for this assumption lies in the following Markov inequality~\cite{Markov:1889}: 
\begin{equation}\label{eq:markov1}
\nm{\na g}{L^\infty(S(K))}\le\dfrac{4m^2R}{r^2}\nm{g}{L^\infty(S(K))}\qquad\text{for all}\quad g\in\mb{P}_m(S(K)).
\end{equation}
Here $\nm{\na g}{L^\infty(S(K))}{:}=\max_{x\in\overline{S(K)}}\nm{\na g(x)}{\ell_2}$. This inequality is proved in~\cite[Lemma 5]{Li:2019}, which is a combination of~\cite[Proposition 11.6]{Wendland:2005} and the fact that $S(K)$ satisfies the uniform interior cone condition~\cite{Adams:2003}, which is a direct consequence of \textbf{Assumption A}.

In~\cite{LiMingTang:2012}, the authors make the following assumption on $S(K)$.
\vskip .5cm
\noindent\textbf{Assumption A'} $S(K)$ is a bounded convex domain and there exists $R$ such that $S(K)\subset B_R$.

By~\cite[Lemma 1.2.2.2 and Corollary 1.2.2.3]{Gris:1985}, \textbf{Assumption A'} implies \textbf{Assumption A}. Under \textbf{Assumption A'}, \textsc{Wilhelmsen}~\cite{Wilhelmsen:1974} proved the following Markov inequality:
\begin{equation}\label{eq:markov1974}
\nm{\na g}{L^\infty(S(K))}\le\dfrac{4m^2}{w(K)}\nm{g}{L^\infty(S(K))}\qquad\text{for all}\quad g\in\mb{P}_m(S(K)),
\end{equation}
where $w(K)$ is the width of $S(K)$, which is the minimum distance between parallel supporting hyperplanes of $S(K)$.

The above Markov inequalities may be viewed as a type of inverse inequality. By the classical inverse inequality for  p-finite element method~\cite[Theorem 4.76]{Schwab:1998} and a simple scaling argument, we may conclude that there exists $C$ independent of the diameter of $S(K)$ but depends on the shape of $S(K)$ such that for all $g\in\mb{P}_m(S(K))$,
\begin{equation}\label{eq:inverse}
\nm{\na g}{L^\infty(S(K))}\le\dfrac{Cm^2}{\text{diam} S(K)}\nm{g}{L^\infty(S(K))}.
\end{equation}
The index $2$ is sharp if $S(K)$ is a locally Lipschitz domain. However, for a cuspidal domain, \textsc{Kro\'o and Szabodos}~\cite{Kroo:2000} proved that if $S(K)$ is a Lip$\gamma$-domain with $0<\gamma<1$\footnote{A typical Lip$\gamma$-domain is the unit $\ell_{\gamma}-$ball, i.e.,$ \set{x\in\mb{R}^d}{\abs{x_1}^{\gamma}+\cdots+\abs{x_d}^{\gamma}\le 1}$ with $0<\gamma<1$.}, then
there exists $C$ depends on the shape of $S(K)$ such that
\begin{equation}\label{eq:cusp}
\nm{\na g}{L^\infty(S(K))}\le\dfrac{Cm^{2/\gamma}}{\text{diam} S(K)}\nm{g}{L^\infty(S(K))},
\end{equation}
where the index $2/\gamma$ is sharp. This would require a larger patch to ensure the reconstruction accuracy, and the reconstruction is less stable than the patch satisfying either \textbf{Assumption A} or \textbf{Assumption A'}.
Moreover, the constant $C$ in~\eqref{eq:inverse} is only known for $S(K)$ with special shape in the literature, e.g., $S(K)$ is an interval~\cite{Rak:2007}, and $S(K)$ is a simplex~\cite{Verfurth:1999, Kroo:1999}. All the prefactors are important for us to derive the realistic conditions that ensure the uniform boundedness of $\Lam(m,\mc{I}(K))$; cf., Lemma~\ref{lema:bd}.

On the other hand, The authors in~\cite{LiMingTang:2012} derived an explicit expression of $C$ that depends on the recursion depth $t$ and the chunkiness parameter $\sigma$ under \textbf{Assumption A'}. It seems that the convexity of the patch is rather restrictive in implementation, particularly for an $L$-shape domain. \textbf{Assumption A} is less restrictive and is easy to check in practice.
\vskip .5cm
\noindent\textbf{Assumption B}
For any $K\in\mc{T}_H$ and $p\in\mb{P}_m(S(K))$,
\[
p|_{\mc{I}(K)} = 0 \text{\qquad implies\qquad} p|_{S(K)}\equiv 0.
\]

This assumption concerns the cardinality of the sampling set $\mc{I}(K)$, which gives the uniqueness and hereby the existence of the solution of the discrete least-squares~\eqref{eq:ls}. \textbf{Assumption B} requires that the cardinality of $\mc{I}(K)$ is at least $\binom{m+d}{d}$ to ensure the unisolvence of the discrete least-squares. A quantitative version of this assumption is
\[
\Lam(m,\mc{I}(K))<\infty,
\]
with
\[
\Lam(m,\mc{I}(K)){:}=\max_{p\in\mb{P}_m(S(K))}\dfrac{\nm{p}{L^\infty(S(K))}}{\nm{p|_{\mc{I}(K)}}
{\ell_\infty}}.
\]

In practical implementation, the positions of the sampling nodes may be slightly perturbed due to the measure error or certain uncertainties~\cite{DeV:2002}. A natural question arises whether the reconstruction is robust with respect to the uncertainties. We shall prove that the reconstruction is stable with respect to small perturbation.

Next we prove some properties of the reconstruction operator $\mc{R}_K$.
\begin{lemma}\label{lema:reconstruction}
If \textbf{Assumption B} holds, then there exists a unique solution of~\eqref{eq:ls2} for any $K\in\mc{T}_H$. Moreover $\mc{R}_K$ satisfies
\[
\mc{R}_K g = g\qquad\text{for all\quad} g\in\mb{P}_m(S(K)).
\]

The stability result is valid for any $K\in\mc{T}_H$ and $g\in C^0(S(K))$, and
\begin{equation}\label{eq:stable}
\nm{\mc{R}_K g}{L^\infty(K)}\le
\Lam(m,\mc{I}(K))\sqrt{\#\mc{I}(K)}\;\nm{g|_{\mc{I}(K)}}
{\ell_\infty}.
\end{equation}
The quasi-optimal approximation property is valid in the sense that
\begin{equation}\label{eq:app-rec}
\nm{g-\mc{R}_Kg}{L^\infty(K)}\le\Lr{1+\Lam(m,\mc{I}(K))}\sqrt{\#\mc{I}(K)}
\inf_{p\in\mb{P}_m(S(K))}\nm{g-p}{L^\infty(S(K))}.
\end{equation}

If \textbf{Assumption A} and \textbf{Assumption B} are valid, then for any $\del\in(0,1)$, there exists
\begin{equation}\label{eq:radius}
\epsilon=\dfrac{\del r^2}{4\Lam(m,\mc{I}(K))m^2R},
\end{equation}
such that for the perturbed sampling set $\wt{I}(K)\subset\mc{I}(K)+B_{\epsilon}(0)$, there exists a unique $\wt{R}_Kg\in\mb{P}_m(S(K))$ satisfying
\begin{equation}\label{eq:perturb1}
\nm{\wt{R}_Kg}{L^\infty(K)}\le\dfrac{\Lam(m,\mc{I}(K))}{1-\del}\sqrt{\#\wt{I}(K)}\,
\nm{g|_{\wt{I}(K)}}{\ell_\infty},
\end{equation}
where $B_{\eps}(0)$ is a ball centered at $0$ with radius $\eps$.

If \textbf{Assumption A'} and \textbf{Assumption B} are valid, then the perturbation result~\eqref{eq:perturb1} remains true with
\[
\eps=\dfrac{\del w(K)}{4\Lam(m,\mc{I}(K))m^2}.
\]
\end{lemma}

The above lemma except the perturbation estimate~\eqref{eq:perturb1} is proved in~\cite[Theorem 3.3]{LiMingTang:2012}, which is crucial for the accuracy of the reconstruction operator $\mc{R}$, more refined estimates on the accuracy of $\mc{R}$ may be found in~\cite[Lemma 4]{Li:2019}. The perturbation estimate~\eqref{eq:perturb1} shows that the set of the sampling nodes $\mc{I}(K)$ perturbed a little bit remains a norming set with a slightly bigger upper bound, i.e., for all $\del\in(0,1)$,
\[
\Lam(m,\wt{I}(K))=\dfrac{\Lam(m,\mc{I}(K))}{1-\del}.
\]
This perturbation estimate has been encapsulated in an abstract form in~\cite[Lemma 2]{Li:2019}, while there is no proof for the perturbed reconstruction operator $\wt{R}$.
\begin{proof}
For any $p\in\mb{P}_m(S(K))$, we let $\abs{p(x^\ast)}=\nm{p|_{\mc{I}(K)}}{\ell_\infty}$, then for any $y\in\wt{I}(K)\subset\mc{I}(K)+B_{\epsilon}(0)$ with $\epsilon$ given by~\eqref{eq:radius}, by Taylor's expansion, we obtain
\[
\abs{p(x^\ast)}\le\abs{p(y)}+\epsilon\nm{\na p}{L^\infty(S(K))}.
\]

By \textbf{Assumption A}, the Markov inequality~\eqref{eq:markov1} is valid, and we obtain
\[
\abs{p(x^\ast)}\le\abs{p(y)}+\dfrac{4m^2R\epsilon}{r^2}\nm{p}{L^\infty(S(K))}.
\]

Using \textbf{Assumption B} and the above two inequalities, we obtain
\begin{align*}
\nm{p}{L^\infty(S(K))}&\le\Lam(m,\mc{I}(K))\nm{p|_{\mc{I}(K)}}{\ell_\infty}
=\Lam(m,\mc{I}(K))\abs{p(x^\ast)}\\
&\le\Lam(m,\mc{I}(K))\abs{p(y)}+\Lam(m,\mc{I}(K))\dfrac{4m^2R}{r^2}\epsilon\nm{p}{L^\infty(S(K))}\\
&\le\Lam(m,\mc{I}(K))\nm{p|_{\wt{I}(K)}}{\ell_\infty}+\del\nm{p}{L^\infty(S(K))},
\end{align*}
which immediately implies
\begin{equation}\label{eq:perturb}
\nm{p}{L^\infty(S(K))}\le\dfrac{\Lam(m,\mc{I}(K))}{1-\del}
\nm{p|_{\wt{I}(K)}}{\ell_\infty}
\qquad\text{for all\quad}\delta\in(0,1).
\end{equation}

Applying the above perturbation estimate to $\wt{R}_K g$, we obtain
\[
\nm{\wt{R}_Kg}{L^\infty(S(K))}\le\dfrac{\Lam(m,\mc{I}(K))}{1-\del}
\nm{\wt{R}_Kg|_{\wt{I}(K)}}{\ell_\infty}.
\]
This also gives the existence and uniqueness of $\wt{R}_K g$.

By the definition of $\wt{R}_K g$, we obtain
\[
\nm{\wt{R}_Kg|_{\wt{I}(K)}}{\ell_\infty}^2\le\nm{\wt{R}_Kg|_{\wt{I}(K)}}{\ell_2}^2
\le\nm{g|_{\wt{I}(K)}}{\ell_2}^2\le\#\wt{I}(K)\nm{g|_{\wt{I}(K)}}{\ell_\infty}^2.
\]
A combination of the above two inequalities and the fact that $\#\wt{I}(K)=\#\mc{I}(K)$ give~\eqref{eq:perturb1}.

If \textbf{Assumption A'} and \textbf{Assumption B} are valid, then we follow exactly the same line that leads to~\eqref{eq:perturb1} except that we use the Markov inequality~\eqref{eq:markov1974} for a convex patch $S(K)$.
\end{proof}

The following lemma ensures the uniform boundedness of $\Lam(m,\mc{I}(K))$. 
\begin{lemma}\label{lema:bd}
If \textbf{Assumption A} holds, then for any $\varepsilon>0$, if
\(
r>m\sqrt{2RH_K(1+1/\varepsilon)},
\)
then we may take
\begin{equation}\label{eq:uniform1}
\Lam(m,\mc{I}(K))=1+\varepsilon.
\end{equation}
Moreover, if
\(
r>2m\sqrt{RH_K},
\)
then we may take
\(
\Lam(m,\mc{I}(K))=2.
\)

If \textbf{Assumption A'} holds, then for any $\varepsilon>0$, if
\(
w(K)>2m^2H_K(1+1/\varepsilon),
\)
then we may take
\begin{equation}\label{eq:uniform2}
\Lam(m,\mc{I}(K))=1+\varepsilon.
\end{equation}
\end{lemma}

The first statement is proved in~\cite[Lemma 5]{Li:2019}, we only prove the second statement~\eqref{eq:uniform2}, which improves~\cite[Lemma 3.5]{LiMingTang:2012}.

\begin{proof}
Let $x^\ast\in\overline{S(K)}$ such that $\abs{p(x^\ast)}=\max_{x\in\overline{S(K)}}\abs{p(x)}$, and $x_\ell\in\mc{I}(K)$ such that
$\abs{x_{\ell}-x^\ast}=\min_{y\in\mc{I}(K)}\abs{y-x^\ast}$. Then
\[
\abs{x_{\ell}-x^\ast}\le H_K/2.
\]

By Taylor's expansion, we have
\[
p(x_\ell)=p(x^\ast)+(x_\ell-x^\ast)\cdot\na p(\xi_x),
\]
where $\xi_x$ lies on the line with endpoints $x^\ast$ and $x_\ell$. This gives
\[
\abs{p(x_\ell)}\le\abs{p(x^\ast)}+\dfrac{H_K}{2}\nm{\na p}{L^\infty(S(K))}.
\]
Using the Markov inequality~\eqref{eq:markov1974}, we immediately have
\[
\nm{p}{L^\infty(S(K))}\le\nm{p|_{\mc{I}(K)}}{\ell_\infty}+\dfrac{2m^2H_K}{w(K)}\nm{p}{L^\infty(S(K))}.
\]
This implies~\eqref{eq:uniform2}.
\end{proof}

If \textbf{Assumption A} holds, then we usually have $R\simeq t H_K$, the estimate~\eqref{eq:uniform1} suggests that $r\simeq m\sqrt{t}H_K$ implies the uniform boundedness of $\Lam(m,\mc{I}(K))$. This means that $S(K)$ cannot be too narrow in certain directions. If \textbf{Assumption A'} holds, then the estimate~\eqref{eq:uniform2} shows that $w(K)\simeq m^2H_K$, which immediately implies $t\simeq m^2$. Both conditions show that a relative large patch is required for the reconstruction. Furthermore, if $S(K)$ is a cuspidal domain, then we may use~\eqref{eq:cusp} to prove that if 
\[
\dfrac{\text{diam}\,S(K)}{H_K}\ge c_0m^{\gamma}(1+1/\eps),
\]
with a constant $c_0$ depending on the shape of $S(K)$, then the bound~\eqref{eq:uniform1} is also valid. This condition indicates
that $S(K)$ has a recursive depth $t\simeq m^{\gamma}$ with $\gamma>2$. An even larger patch is required to ensure the stability. This may occur for domain $D$ with complicated boundary or rough boundary.

It remains to find an upper bound for $\#\mc{I}(K)$. This is a direct consequence of \textbf{Assumption A} or \textbf{Assumption A'} and the shape regularity of $\mc{T}_H$.
\begin{lemma}\label{lema:cardinality}
If $\mc{T}_H$ is shape-regular and \textbf{Assumption A} or \textbf{Assumption A'} is valid, then
\begin{equation}\label{eq:upp}
\#\mc{I}(K)\le\Lr{\sigma R/H_K}^d.
\end{equation}
\end{lemma}

\begin{proof}
For any element $K\in\mc{T}_H$, using \textbf{Assumption A} or \textbf{Assumption A'} and note that there is only one sample point inside each element $K$, we obtain
\[
\#\mc{I}(K)\abs{K}\le\text{Vol\;}B_d(R),
\]
where $\text{Vol\;}B_d(R)$ stands for the volume of a $d$-dimensional ball with radius $R$.

Using the shape regularity of $\mc{T}_H$, we obtain
\[
\abs{K}\ge\text{Vol\;}B_d(\rho_K)\ge\sigma^{-d}\text{Vol\;}B_d(h_K).
\]

A combination of the above two inequalities yields~\eqref{eq:upp}.
\end{proof}

It is clear that the upper bound~\eqref{eq:upp} is independent of the way for construction $S(K)$. For the two ways based on either the Moore neighbor or the von Neumann neighbor, we
have $\#\mc{I}(K)\simeq t^d$ with $t$ the recursion depth, which is consistent with the corresponding upper bound proved in~\cite[Lemma 3.4]{LiMingTang:2012} and~\cite[Lemma 6]{Li:2019}, in which the two-dimensional problem has been dealt with.
Both Lemma~\ref{lema:bd} and Lemma~\ref{lema:cardinality} require that $\#\mc{I}(K)$ should be quite large, equivalently, $S(K)$ is large, so that the uniform boundedness of $\Lam(m,\mc{I}(K))$ is valid. In numerical tests below, we observe that the method still works quite well even when $\#\mc{I}(K)$ is far less than the theoretical threshold. We refer to~\cite{LiMingTang:2012} for a list of the size of $S(K)$ and the upper bound of $\max_{K\in\mc{T}_H}\Lam(m,\mc{I}(K))$.

Based on the above three lemmas, we are ready to estimate~\eqref{eq:decmod}.
\begin{lemma}
If \textbf{Assumption A} or \textbf{Assumption A'}  and \textbf{Assumption B} are valid and the effective matrix $\mc{A}_{ij}\in C^{m+1}(D)$, then there exists $C$ that depends on $\nm{\mc{A}_{ij}}{C^{m+1}(D)},m,R,r,\gamma$ and $t$ but independent of $H$.
\begin{equation}\label{eq:estmatrix}
e(\hmm)\le C\Lr{H^{m+1}+e_1(\hmm)},
\end{equation}
where $e_1(\hmm){:}=\max_{x\in\mc{I}(K),K\in\mc{T}_H}\nm{(\mc{A}-A_H)(x)}{F}$.
\end{lemma}

It is worthwhile to mention that $e_1(\hmm)$ is the so-called estimating error in HMM~\cite{EE:2003}. There are many works devoted to bounding $e_1(\hmm)$ and developing new algorithms to improve the
estimates; see; e.g.,~\cite{EMingZhang:2005, Yue:2007, Blanc:2010, Gloria:2011, Olof:2016, Gloria:2016, Mourrat:2019} and the references therein.

\begin{proof}
We start from the decomposition~\eqref{eq:decmod}. For each $K\in\mc{T}_H$, using \textbf{Assumption A} and \textbf{Assumption B} or \textbf{Assumption A'} and \textbf{Assumption B}, Lemma~\ref{lema:bd}, Lemma~\ref{lema:cardinality} and ~\eqref{eq:app-rec}, we obtain, on each element $K\in\mc{T}_H$, there exists $C$ that depends on $t,\sigma,r,R,d$ and $\nm{\mc{A}_{ij}}{C^{m+1}(D)}$ such that
\[
\nm{\mc{A}_{ij}-(\mc{R}_K\mc{A})_{ij}}{L^\infty(K)}\le C\inf_{p\in\mb{P}_m(S(K))}
\nm{\mc{A}_{ij}(x)-p}{L^\infty(S(K))}\le CH^{m+1}.
\]
Summing up all $i,j=1,\cdots,d$ and $K\in\mc{T}_H$ we obtain
\begin{equation}\label{eq:accuracy-rec}
\max_{x\in D}\nm{\mc{A}(x)-\mc{R}\mc{A}(x)}{F}\le CH^{m+1}.
\end{equation}

Using~\eqref{eq:stable}, we obtain, for each $K\in\mc{T}_H$,
\[
\nm{\mc{R}_K\Lr{\mc{A}_{ij}-(\overline{A}_H)_{ij}}}{L^\infty(K)}\le\Lam(m,\mc{I}(K))\sqrt{\#\mc{I}(K)}
\nm{\Lr{\mc{A}_{ij}-(A_H)_{ij}}|_{\mc{I}(K)}}{\ell_\infty}.
\]
Summing up all $i,j=1,\cdots,d$, we obtain
\[
\max_{x\in K}\nm{\mc{R}_K(\mc{A}-A_H)(x)}{F}^2
\le\Lam^2(m,\mc{I}(K))\#\mc{I}(K)\max_{x\in\mc{I}(K)}\nm{\mc{A}(x)-A_H(x)}{F}^2.
\]
Using \textbf{Assumption A} and \textbf{Assumption B} or \textbf{Assumption A'} and \textbf{Assumption B}, Lemma~\ref{lema:bd} and Lemma~\ref{lema:cardinality}, there exists $C$ depends on $R,r,t$ and $\sigma$ such that
\[
\max_{x\in D}\nm{\mc{R}(\mc{A}-A_H)}{F}\le Ce_1(\hmm),
\]
which together with~\eqref{eq:accuracy-rec} gives~\eqref{eq:estmatrix} and finishes the proof.
\end{proof}

Substituting the estimate~\eqref{eq:estmatrix} into Lemma~\ref{lema:1staccuracy}, we obtain the main result of this part.
\begin{theorem}\label{thm:main}
Let $u_0$ be the homogenized solution with $\nm{u_0}{H^{ l+1}(D)}<\infty$. If \textbf{Assumption A} or \textbf{Assumption A'} and \textbf{Assumption B} hold, and $e_1(\hmm)<\al/4$, then there
exists a unique solution $u_h$ of Problem~\eqref{eq:online} that satisfying
\begin{equation}\label{eq:enerfinal}
\nm{\na(u_0-u_h)}{L^2(D)}\le C\Lr{h^l+H^{m+1}+e_1(\hmm)}.
\end{equation}

Moreover, if the solution of the auxiliary problem~\eqref{eq:auxprob} admits a unique solution $\phi_g$ satisfying the regularity estimate $\nm{\phi_g}{H^2(D)}\le C\nm{g}{L^2(D)}$, then
\begin{equation}\label{eq:l2final}
\nm{u_0-u_h}{L^2(D)}\le C\Lr{h^{l+1}+H^{m+1}+e_1(\hmm)}.
\end{equation}
\end{theorem}

In view of the above error estimates~\eqref{eq:enerfinal} and~\eqref{eq:l2final}, we may have $H\simeq h^{l/(m+1)}$ or $H\simeq h^{(l+1)/(m+1)}$. For a fixed $h$, one may increase $m$ to decrease the cost in the offline stage. This suggests that higher order reconstruction is preferred to save cost while without loss of accuracy, which is confirmed by the tests in the next section. 

The error estimate for HMM-LS method in~\cite{LiMingTang:2012} reads as
\begin{equation}\label{eq:oldest}
\left\{\begin{aligned}
\nm{\na(u_0-\wt{u}_h)}{L^2(D)}&\le C\Lr{h^l+h^{m+1}+e_1(\hmm)},\\
\nm{u_0-\wt{u}_h}{L^2(D)}&\le C\Lr{h^{l+1}+h^{m+1}+e_1(\hmm)},
\end{aligned}\right.
\end{equation}
where $\wt{u}_h$ is the solution of HMM-LS method. For a fixed $h$, the above error estimate indicates that $m\simeq l$ to balance the error provided that $e_1(\hmm)$ is sufficiently small, the number of the cell problems in this case is the total number of the vertices of the online triangulation $\tau_h$. We shall compare these two methods in the next part.
\section{Numerical Results}\label{sec:test}
In this section, we report a few numerical examples to test the accuracy and efficiency of the proposed offline-online method. The examples include the problems with locally periodic coefficients, with almost periodic coefficients and with random coefficients. We also test problems posed on L-shape domain in two dimension and three dimension for which the solutions are usually nonsmooth. 

In all the tests, the offline triangulations $\mc{T}_H$ are uniform partitions of  $D$, and we list the threshold number ($N_{\text{lowest}}$ is defined in \S~\ref{sec:method}) of the sampling points for different orders of reconstruction in Table~\ref{table:number}, the number $\#\mc{I}(K)$ should be slightly larger than the number that lists in the table when $K$ abuts the boundary of the domain. For $d=3$, we only report $m=3$ because this is the only case we test in this part, the examples for the lower order reconstruction may be found in~\cite{Huang:2018}. 
%
\begin{table}[h]
	\caption{The number of the threshold number for the sampling point in reconstruction.}
	\label{table:number}
	\vspace{.2cm}
	\centering
	\begin{tabular}{|c||c|c|c|}
		\hline
		&$m=1$& $m=2$ &$m=3$\\
		\hline
		$d=2$ & 5& 7&13\\	
		\hline	
		$d=3$&NONE&NONE&27\\
		\hline
\end{tabular}
\end{table}

Noted that the theory covers the case when  $S(K)$ consists of polygons or polytopes, we refer to~\cite{Li:2019} for the demonstration of such patches. It is worth pointing out that the nonuniform $\mc{T}_H$ is useful when the shape of $D$ is very complicate. In that case, the sampling points near the boundary has to be dense to ensure the accuracy of the reconstruction. We refer to~\cite{Huang:2018} for the examples of nonuniform patch, which may further reduce the number of the cell problems and makes $e(\text{MOD})$ more evenly distributed over the whole domain, which is useful for adaptivity. 

The finite element solvers except Example~\ref{ex:3d} are carried on FreeFem++ toolbox~\cite{Hecht:2012}\footnote{https://freefem.org/}, and the test for Example~\ref{ex:3d} is performed in a parallel hierarchical grid platform (PHG)~\cite{Zhang:2009}\footnote{http://lsec.cc.ac.cn/phg/}. The error of the method is calculated by the relative error measured in $H^1$ norm and $L^2$ norm:
\[
\dfrac{\nm{\na(u_0-u_h)}{L^2(D)}}{\nm{\na u_0}{L^2(D)}}\qquad\text{and}\qquad\dfrac{\nm{u_0-u_h}{L^2(D)}}{\nm{u_0}{L^2(D)}}.
\]
The \textit{exact} homogenized solution $u_0$ is generated by discretization~\eqref{eq:homo} with $\mb{P}_2$ FEM over a very refined $500\x 500$ mesh in most cases unless otherwise stated. In all the tests we set $\eps=10^{-6}$.
\subsection{Problems with local periodic coefficient in two dimension}
We consider an example with locally periodical coefficient in $d=2$, which is taken from~\cite{MingYue:2006}.
\begin{example}\label{ex:2dtang}
\[
\left\{\begin{aligned}
     -\div(\a\na u^{\epsilon}(x))&=f(x),\quad\quad   &&\text{in} \quad D,\\
     u^{\epsilon}(x)&=0, \quad &&\text{on} \quad \pa D,
   \end{aligned}\right.
\]
where $D =(0,1)^2$, $f(x)=1$ and
\[
     \a(x)=\frac{(2.5+1.5\sin(2\pi x_1))(2.5+1.5\cos(2\pi x_2))}{(2.5+1.5\sin(2\pi x_1/\eps))(2.5+1.5\sin(2\pi x_2/\epsilon))}1_{2\x 2},
\]
where $1_{2\x 2}$ denotes the $2-$by$-2$ identity matrix. The homogenization problem is
\begin{equation}\label{eq:2dhomo}
\left\{\begin {aligned}
     -\div(\mc{A}(x)\na u_0(x))&=f(x),\quad\quad   &&\text{in} \quad D,\\
     u_0(x)&=0, \quad \quad&&\text{on} \quad \pa D.
   \end {aligned}\right.
\end{equation}
A direct calculation gives the following analytical expression of $\mc{A}$:
\begin{equation}\label{eq:expeff}
  \mc{A}(x_1,x_2)=\dfrac15(2.5+1.5\sin(2\pi x_1))(2.5+1.5\cos(2\pi x_2))1_{2\x 2}.
\end{equation}
\end{example}

The offline triangulation $\mc{T}_H$ consists of a uniform $Q\x Q$ squares. To obtain the online triangulation $\tau_h$, we firstly triangulate $D$ into a uniform $N\x N$ squares and secondly divide each square into two sub-triangles along its diagonal with positive slope. To study the effect of the reconstruction, we use the analytical expression~\eqref{eq:expeff} of $\mc{A}$ for reconstruction so that $e_1(\text{MOD})=0$.  We denote this numerical solution by $u_h^0$. Fig.~\ref{fig:periodeHMM} presents the accuracy of the offline computation, which corroborates the estimate~\eqref{eq:estmatrix}.
\begin{figure}[h]
	\centerline{\includegraphics[width=8cm]{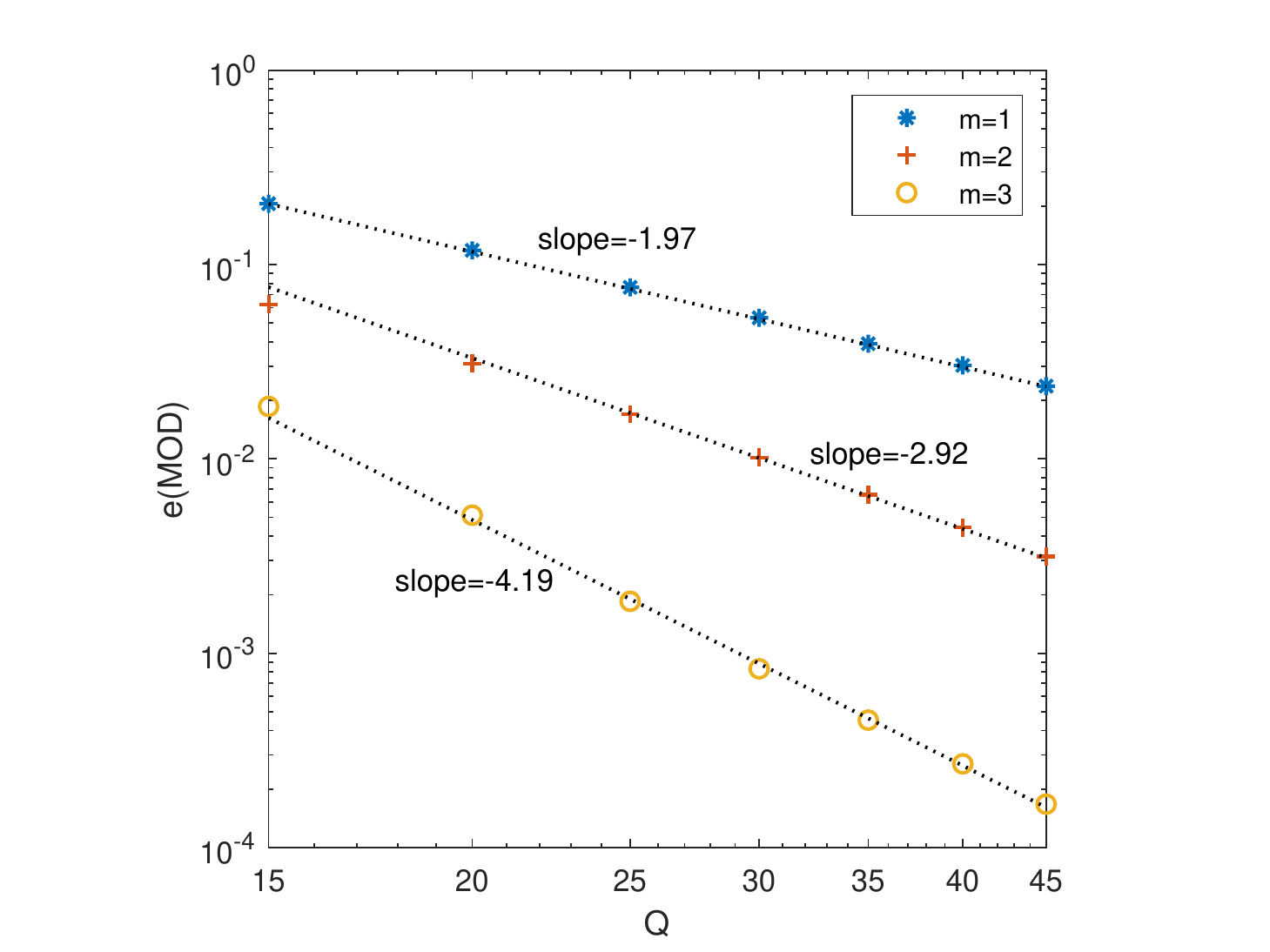}}
	\caption{$e(\text{MOD})$ for different reconstruction orders.}
	\label{fig:periodeHMM}
\end{figure}

To study the convergence rate of the method, it is more convenient to reshape the error estimates in Theorem~\ref{thm:main} in terms of $N$ and $Q$ as follows.
\begin{equation}\label{eq:refinement}
\left\{\begin{aligned}
\nm{\na(u_0-u_h)}{L^2(D)}&\le C(N^{-l}+Q^{-(m+1)}),\\
\nm{u_0-u_h}{L^2(D)}&\le C(N^{-(l+1)}+Q^{-(m+1)}).
\end{aligned}\right.
\end{equation}
Balancing the discretization error and the reconstruction error, we set $Q$ as
\[
Q=\left\{\begin {aligned}
     \mc{O}(N^{\frac{l}{m+1}}), &\quad H^1 \text{ error}, \\
     \mc{O}(N^{\frac{l+1}{m+1}}),&\quad  L^2  \text{ error}.
   \end{aligned}\right.
\]
Following this refinement strategy, we plot the relative $H^1$ error and $L^2$ error in Fig.~\ref{fig:periodrate}, which are consistent with the estimates~\eqref{eq:refinement}.
\begin{figure}[h]
	\begin{minipage}{0.48\linewidth}
		\centering
		\includegraphics[width=7cm]{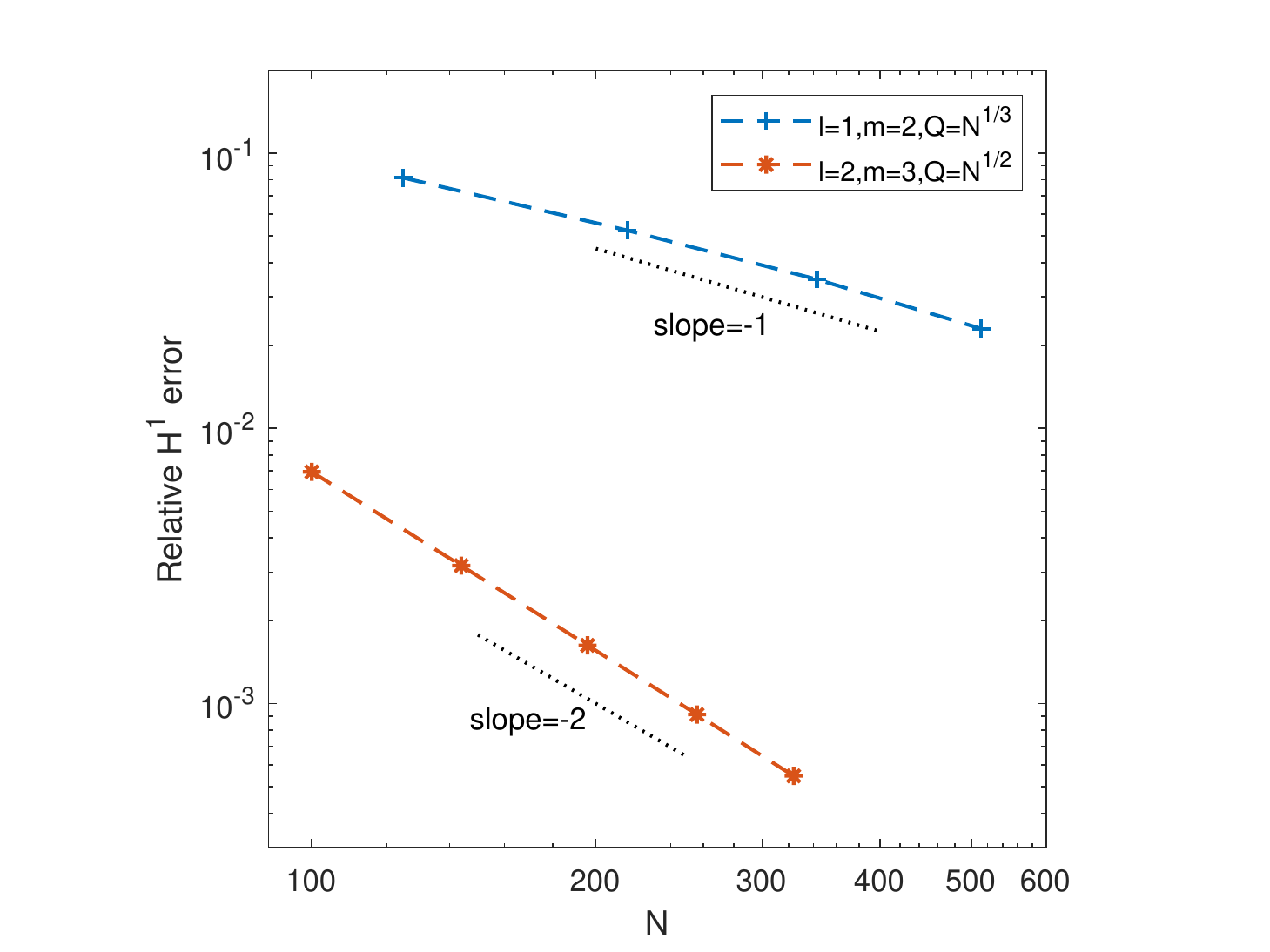}
		\center{(a). Rate of convergence in $H^1$.}
	\end{minipage}
	\hfill
	\begin{minipage}{0.48\linewidth}
		\centering
		\includegraphics[width=7cm]{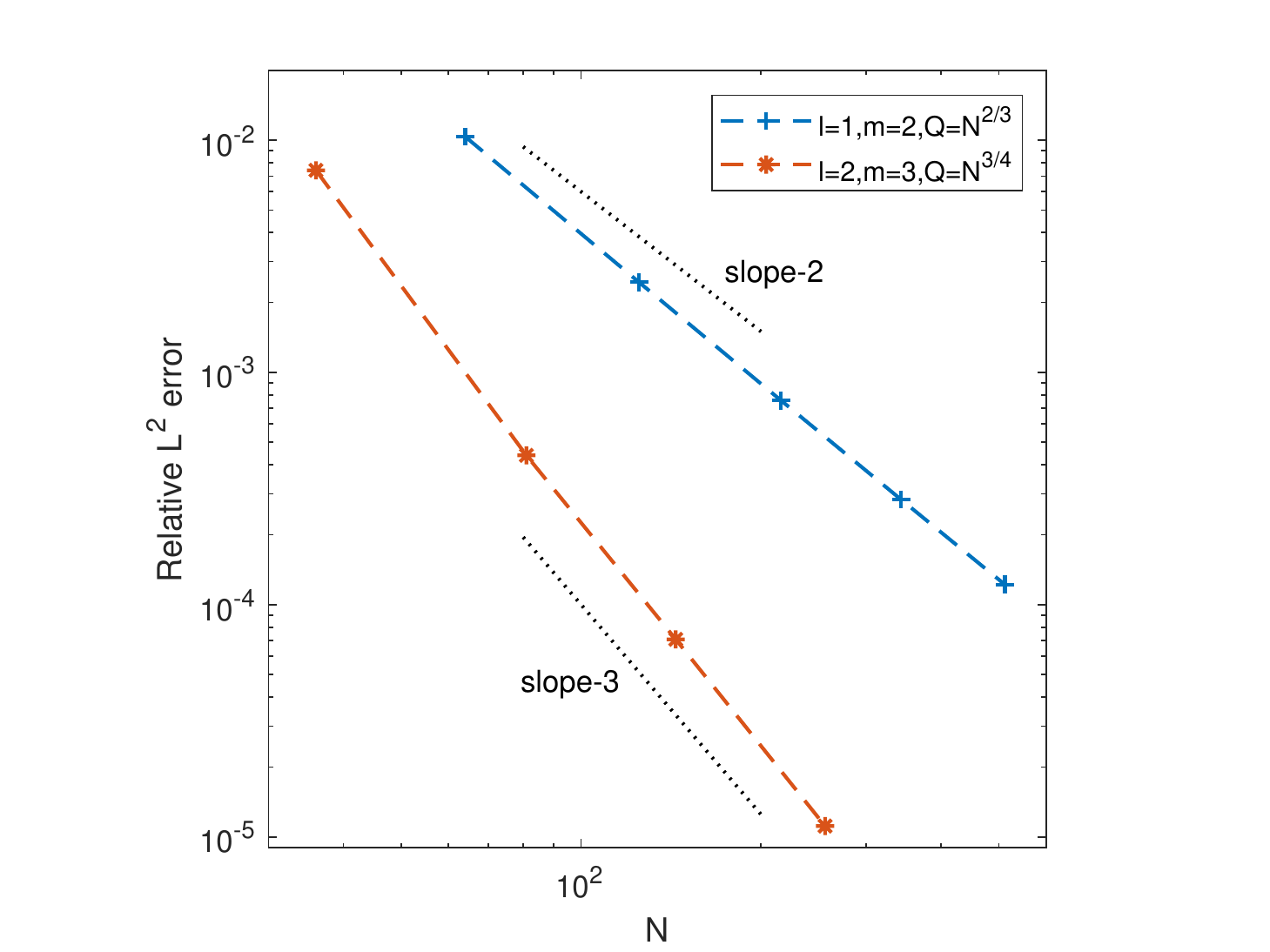}
		\center{(b). Rate of convergence in $L^2$.}
	\end{minipage}
	\caption{Rates of convergence for Example~\ref{ex:2dtang}.}
	\label{fig:periodrate}
\end{figure}

We may also fix the accuracy of the online solver and compare the effect of the reconstruction with different orders in the offline computation. Particularly we choose the online solver as $\mb{P}_2$ FEM with the meshsize parameter $N=100$. Fig.~\ref{fig:perioderr} clearly shows that the higher-order reconstruction is  more accurate with less cost.
\begin{figure}[h]
	\begin{minipage}{0.48\linewidth}
		\centering
		\includegraphics[width=7cm]{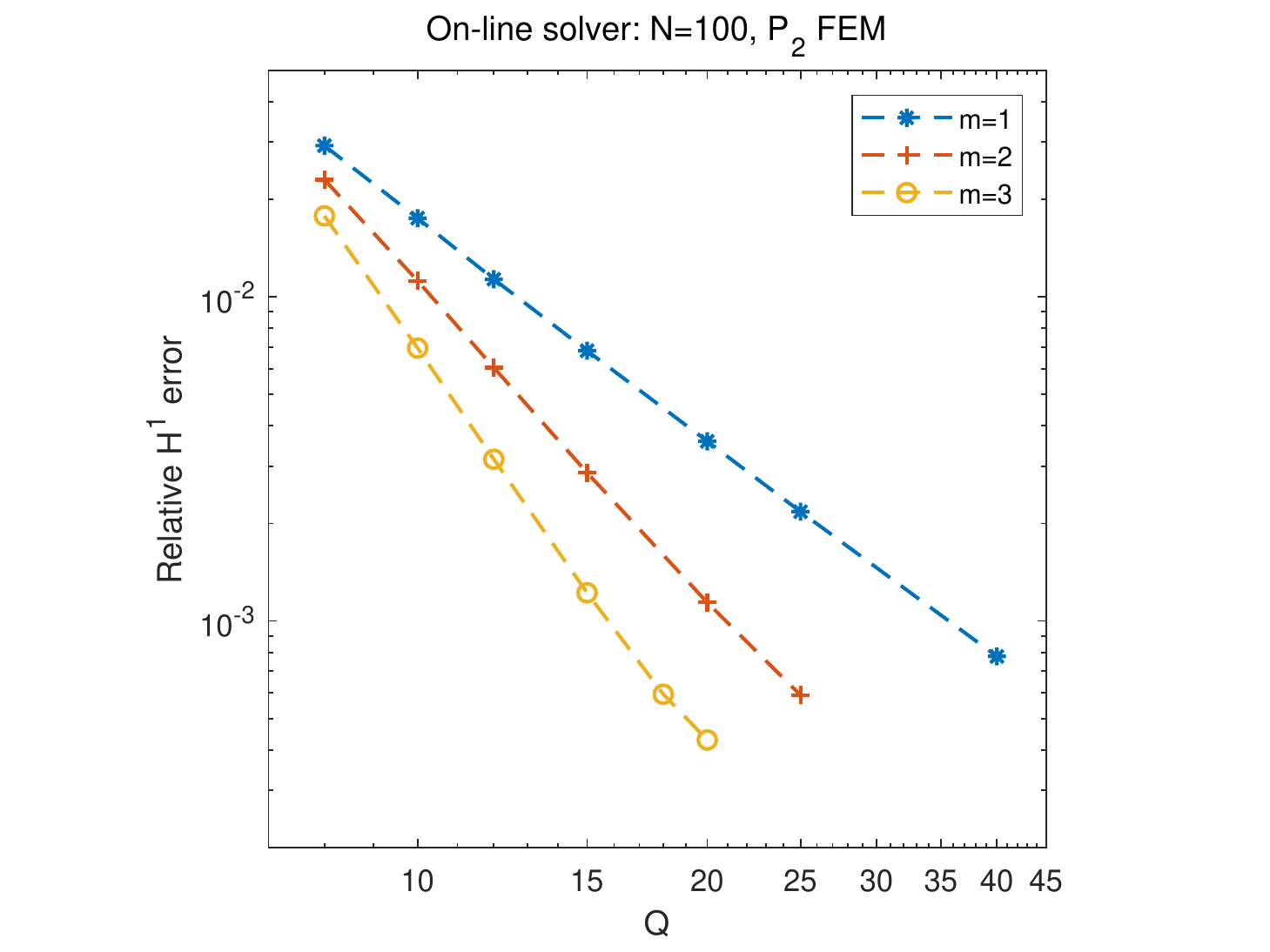}
\center{(a). The $H^1$ error with different reconstruction orders.}
	\end{minipage}
	\hfill
	\begin{minipage}{0.48\linewidth}
		\centering
		\includegraphics[width=7cm]{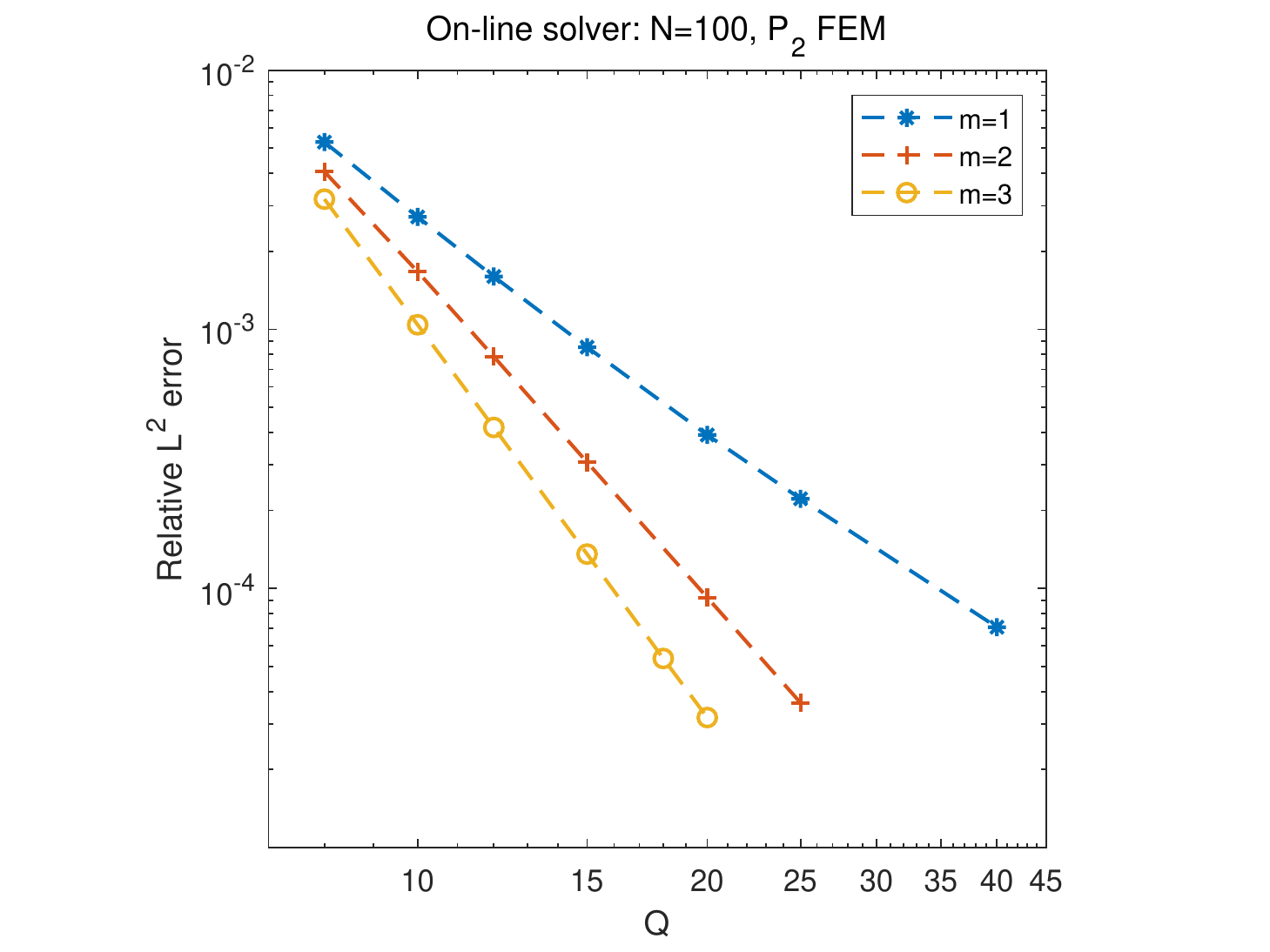}
\center{(b). The $L^2$ error with different reconstruction orders.}
	\end{minipage}
	\caption{The effect of $Q$ and the reconstruction orders in Example~\ref{ex:2dtang}; Online solver is $\mb{P}_2$ FEM and $N=100$.}
	\label{fig:perioderr}
\end{figure}

Next, we also fix the accuracy of the online solver and compare the number of the cell problems and the running time with reconstructions of different orders. The online solver is  still $\mb{P}_2$ FEM with the meshsize parameter $N=100$. Instead of using the analytical expression of $\mc{A}$ for reconstruction, we solve the periodic cell problems
\begin{equation}\label{eq:percell}
\left\{\begin{aligned}
-\div(\a\na \vv_i(x))&=0\qquad&&\text{in\quad}I_\eps,\\
\vv_i-x_i&\text{\quad is periodic}\qquad&&\text{on\quad}\pa I_\eps,
\end{aligned}\right.
\end{equation}
with $\mb{P}_2$ FEM, and employ~\eqref{eq:effective} to obtain $A_H$. We denote the reconstructed solution as $u_h^1$, which is very close to $u_h^0$ in the sense
\[
\dfrac{\nm{u_h^1-u_h^0}{L^2(D)}}{\nm{u_h^0}{L^2(D)}}\le 10^{-8}.
\]
This means that the $e_1(\text{MOD})$ is very small, though nonzero. Table~\ref{table:PeriodEx1} shows that the main cost comes from solving the cell problems, and higher-order reconstruction is more efficient in terms of both the accuracy and the efficiency.
\begin{table}[h]
	\caption{Comparison among reconstruction of different orders with the same online 
	  solver ($N=100$, $\mb{P}_2$ FEM).}
	\label{table:PeriodEx1}
	\vspace{.2cm}
	\centering
	\begin{tabular}{|c||c|c|c|c|c|}
		\hline
		& Q &$\sharp$ cell problems & Relative $H^1$ error &\begin{tabular}{c}Time\\(cell problems)\end{tabular}&Total Time\\
		\hline
		$m=1$ & 45& 2025&6.03e-4& 348.65s&363.27s\\	
		\hline	
		$m=2$ & 25& 625 &5.92e-4&106.97s&110.76s\\	
		\hline	
		$m=3$ & 18 & 324 &5.92e-4&54.5s&58.23\\	
		\hline	
	\end{tabular}
\end{table}

In the last test, we compare the offline-online method with HMM-LS method in~\cite{LiMingTang:2012}. The $H^1$ error bound in~\eqref{eq:oldest} may be rewritten as
\[
\nm{\na (u_0-\wt{u}_h)}{L^2(D)}\le C(N^{-l}+N^{-(m+1)}+e_1(\text{MOD})).
\]
We follow the setup in the last test to solve the cell problems so that $e_1(\text{MOD})$ is negligible in the above estimate. We take $l=2$ and $m=1$ to balance the error so that
\[
\nm{\na (u_0-\wt{u}_h)}{L^2(D)}\le C\,N^{-2}. 
\]
In the offline-online method, we use $\mb{P}_2$ FEM  as the macroscopic solver and choose the reconstruction order $m=2$ and $m=3$. To balance the error of ~\eqref{eq:refinement}, we take $Q=2N^{2/3}$ for $m=2$ and $Q=2.5N^{1/2}$ for $m=3$ and we obtain $\nm{\na (u_0-u_h)}{L^2(D)}\le C\,N^{-2}$. It seems that the relative $H^1$  errors for both methods reported in Fig.~\ref{fig:compare12} are comparable, which is consist with the theoretical estimates. The number of the cell problems and the running time plotted in Fig.~\ref{fig:compare12} demonstrate that the offline-online method is more efficient. This is easily understood because the number of the cell problems in the offline-online method is of $\mc{O}(N^{4/3})$ for $m=2$ and is of $\mc{O}(N)$ for $m=3$, while the number of the cell problems in HMM-LS is of $\mc{O}(N^2)$.
 \begin{figure}[h]\centering
\begin{minipage}{0.48\linewidth}
  \centerline{\includegraphics[width=7cm]{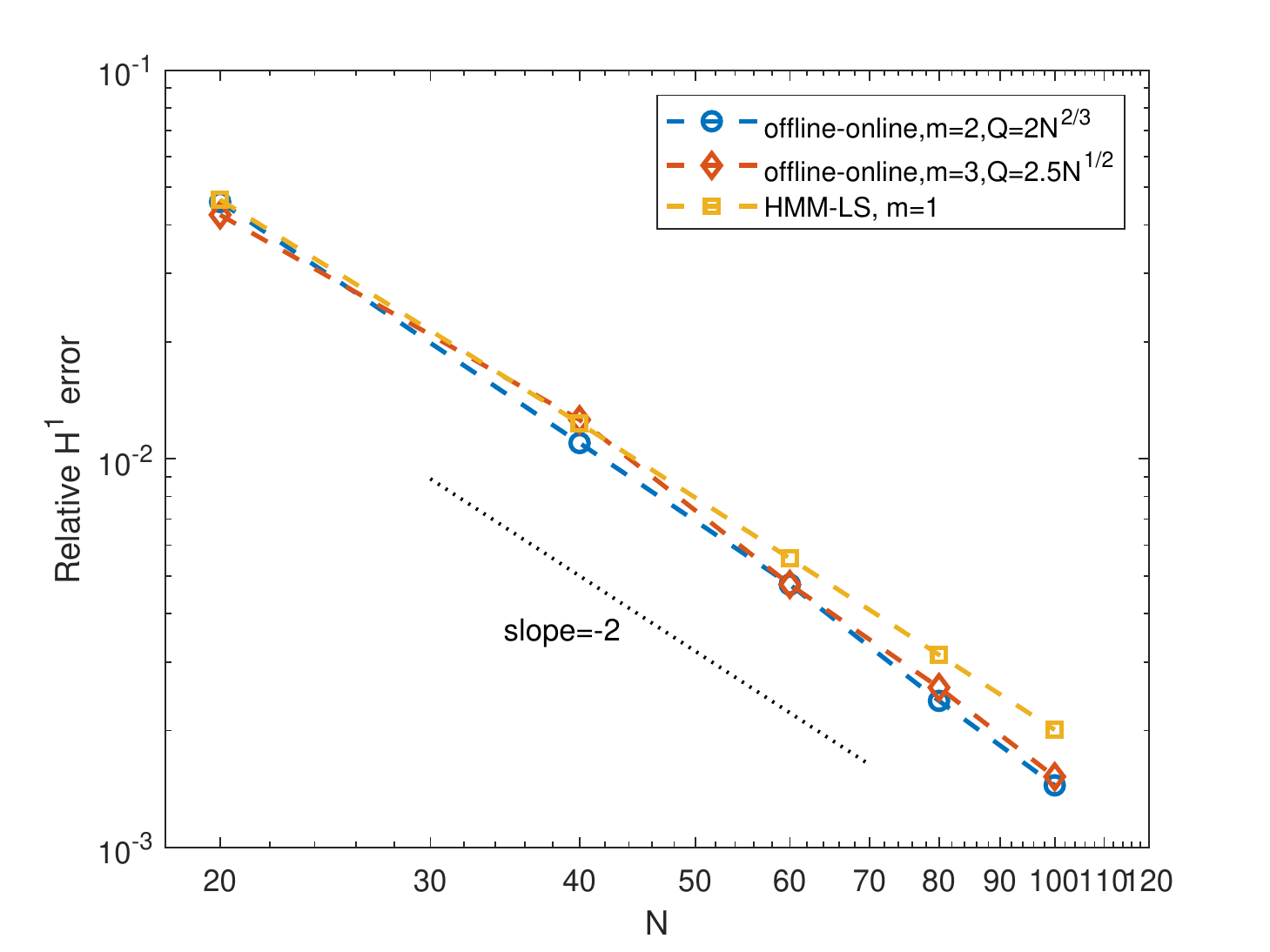}}
  \center{(a). Relative $H^1$ error.}
\end{minipage}
\hfill
\begin{minipage}{0.48\linewidth}
  \centerline{\includegraphics[width=7cm]{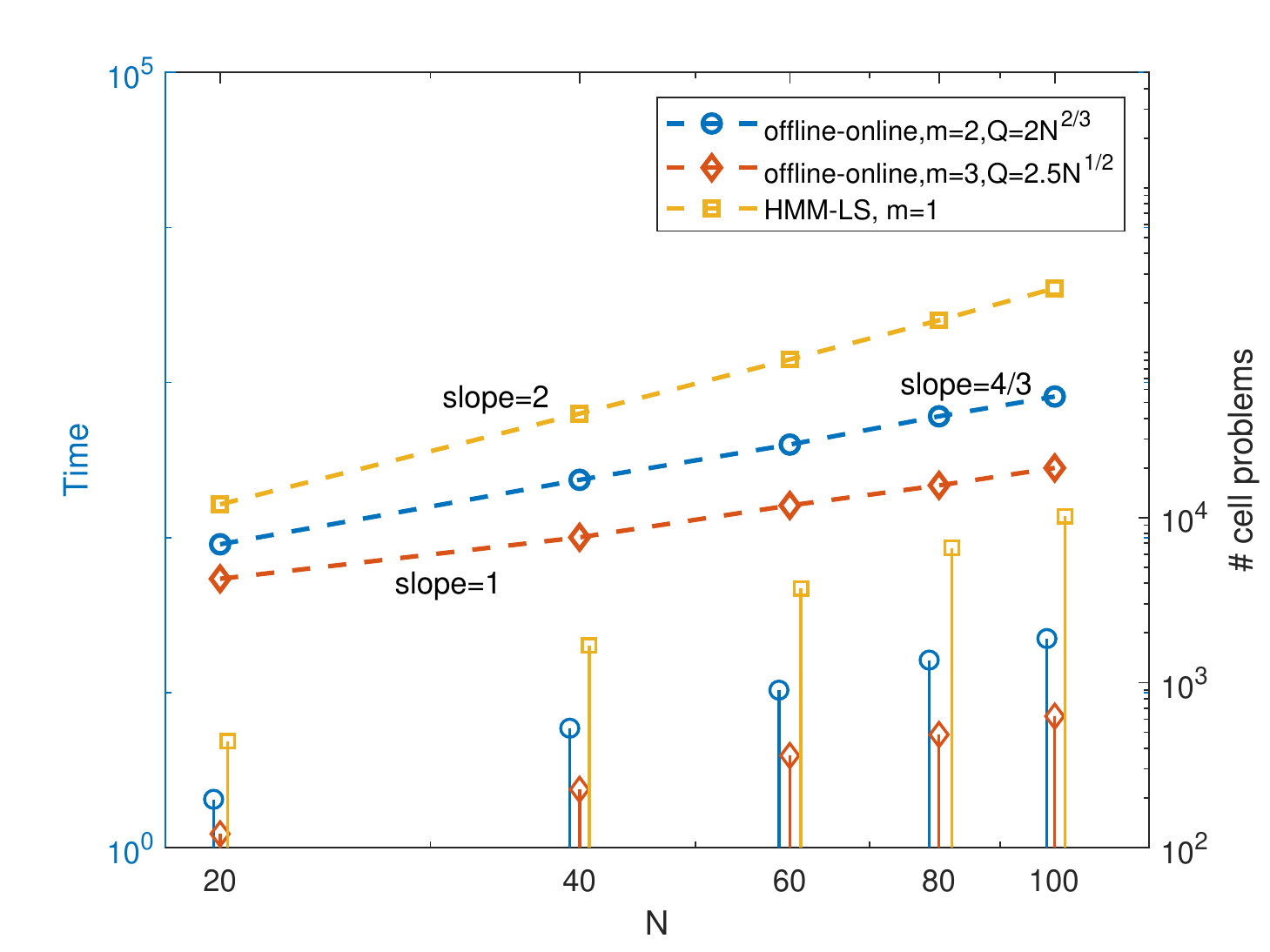}}
  \center{(b).  Running time and the number of the cell problems.}
\end{minipage}
\caption{Comparison between the offline-online method and  HMM-LS method in~\cite{LiMingTang:2012}.} \label{fig:compare12}
\end{figure}
\subsection{Problems posed on L-shape domain in $d=2,3$}
In this part, we test the offline-online strategy for a problem posed on L-shape domain. It is well-known that the adaptive 
strategy has to be used in the macroscopic solver because the solution $u_0$ is nonsmooth~\cite{Gris:1985}.
\begin{example}\label{ex:2dada}
The boundary value problem is the same with Example~\ref{ex:2dtang} except that the domain $D= (-1,1)^2\setminus[0,1]\times[-1,0]$. We let
\[
  u_0(x) = (x_1^2+x_2^2)^{1/3}.
\]
The inhomogeneous boundary condition $g(x)=u_0(x)|_{x\in\pa D}$ celland the source term $f$ is computed from the homogenized problem~\eqref{eq:2dhomo}$_1$.
\end{example} 

In the offline stage, the periodic cell problems~\eqref{eq:percell} are solved by $\mb{P}_2$ FEM over a mesh of size $50\times 50$. The third order reconstruction is used to obtain $A_H$. In the online stage, $\mb{P}_1$ FEM is used as the macroscopic solver. Then the error estimate roughly reads as
\begin{equation}\label{eq:2dada}
  \nm{\na(u_0-u_h)}{L^2(D)}\le C\Lr{\text{DOF}^{-1/2}+Q^{-4}+e_1(\text{MOD})},
\end{equation}
where DOF is the total degrees of freedom in the online computation
Given the error tolerance (TOL), in the offline stage we need to require
\[
   Q\lesssim\text{TOL}^{-1/4}\qquad\text{and}\qquad e_1(\text{MOD})\lesssim\text{TOL}.
\]
%

Firstly we set TOL$=10^{-2}$. The triangulation $\mc{T}_H$ is plot in Fig.~\ref{fig:l2d}a, we have $972$ sampling points in total. Under these settings,  $e(\text{MOD})=1.54E-02$ in the offline stage.
 let $\mc{T}_H$ be the initial mesh of the online computation. The mesh is refined by the strategy taken from~\cite{Hecht:2012}, and we plot the resulting mesh after $8$ iterations in Fig.~\ref{fig:l2d}b. 
\begin{figure}[htbp]\centering
	\begin{minipage}{0.48\linewidth}
	\centerline{\includegraphics[width=8cm]{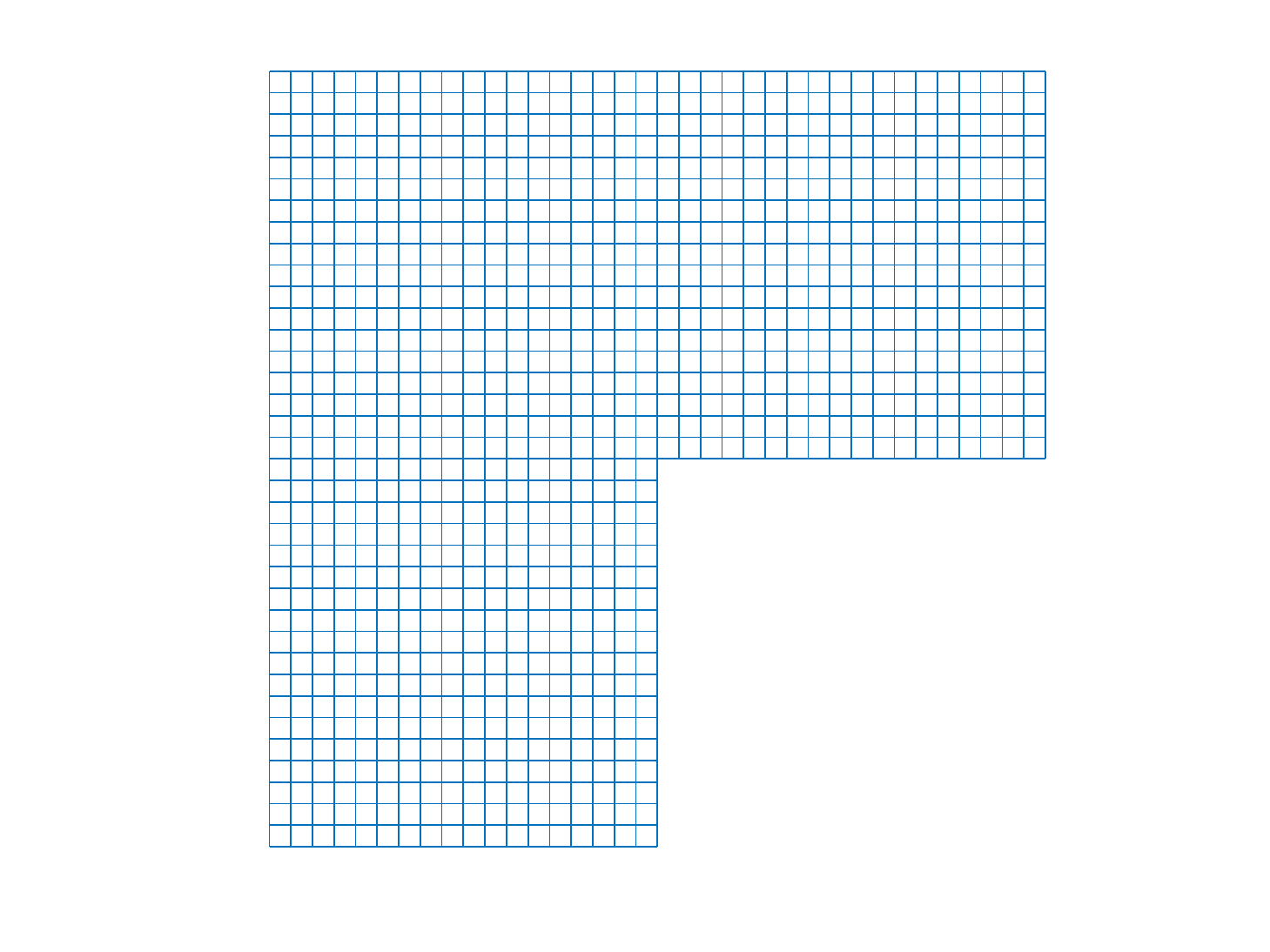}}
	\center{(a). Offline sampling mesh $\mathcal{T}_H$.}
\end{minipage}\hfill
\begin{minipage}{0.48\linewidth}
	\centerline{\includegraphics[width=8cm]{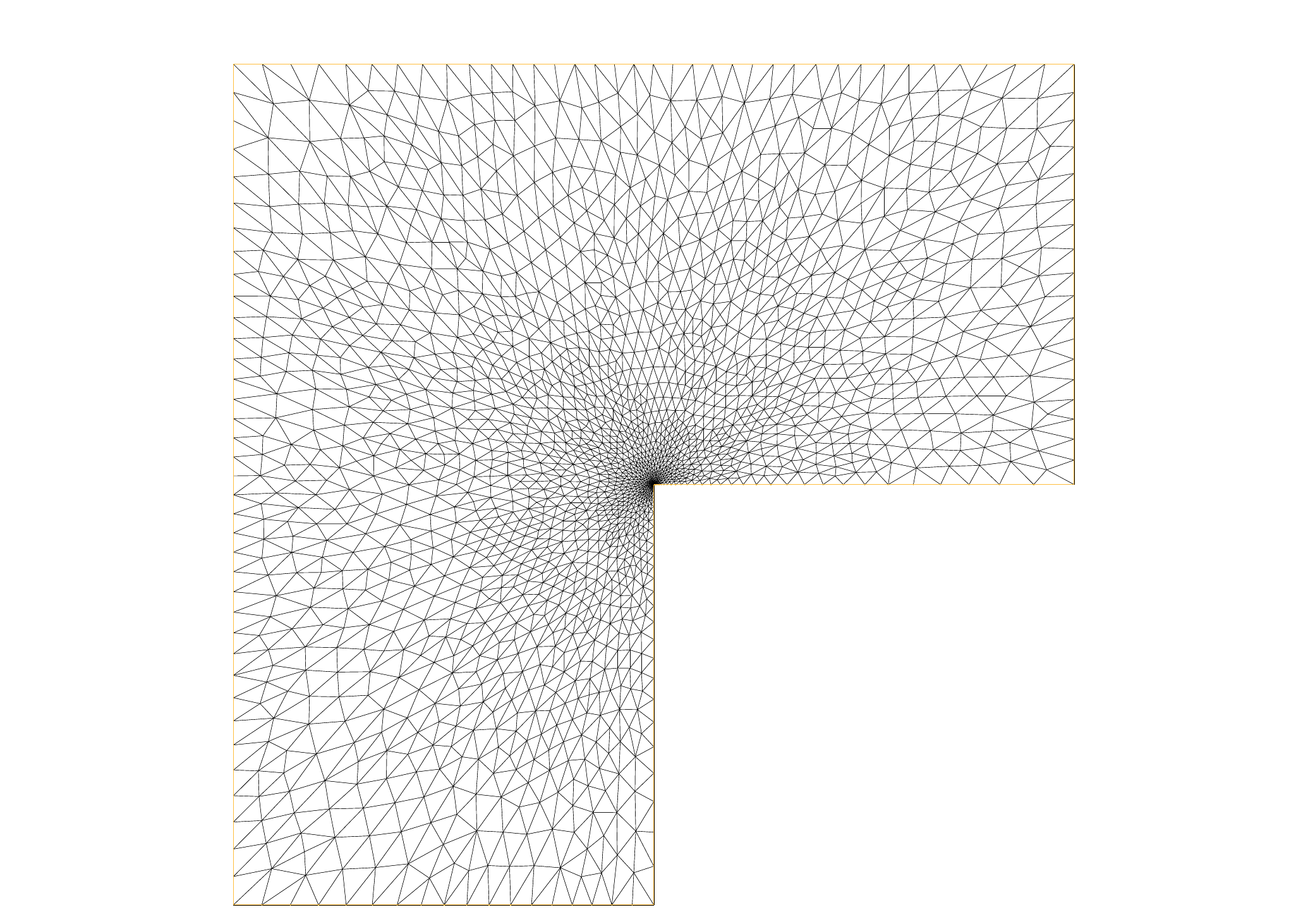}}
	\center{(b). Online adaptive mesh $\tau_h$ after $8$ iterations.}
	\end{minipage}
	\caption{Offline and online triangulations for L-shaped domain.}\label{fig:l2d}
\end{figure}	
The relative $H^1$ error is plot in Fig.~\ref{fig:2dada}a, and we observe that the error decays with optimal
rate of convergence $\mc{O}(\text{DOF}^{-1/2})$.
 \begin{figure}[h]\centering
\begin{minipage}{0.48\linewidth}
  \centerline{\includegraphics[width=7cm]{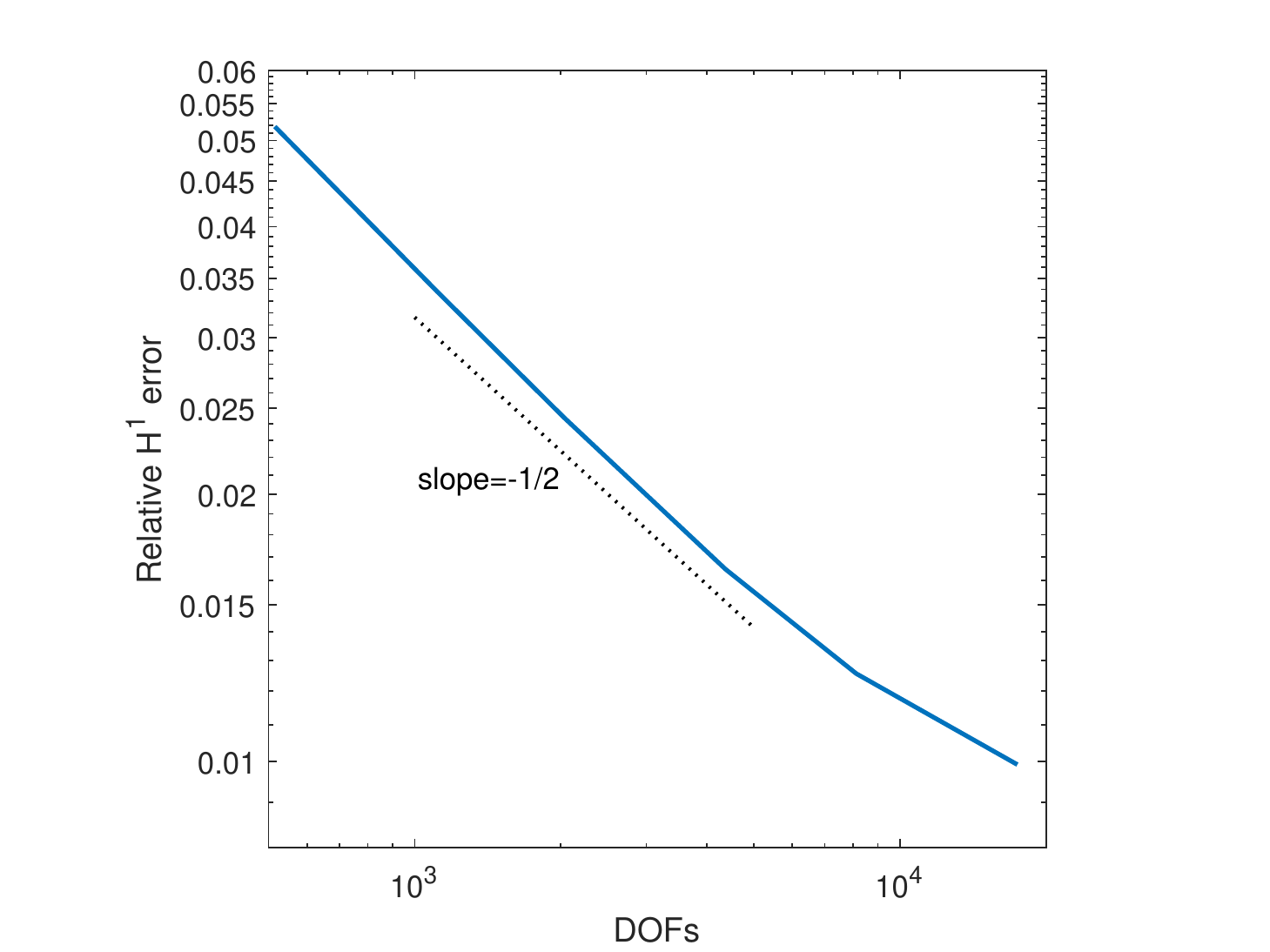}}
  \center{(a). Tolerance$=10^{-2}$.}
\end{minipage}
\hfill
\begin{minipage}{0.48\linewidth}
  \centerline{\includegraphics[width=7cm]{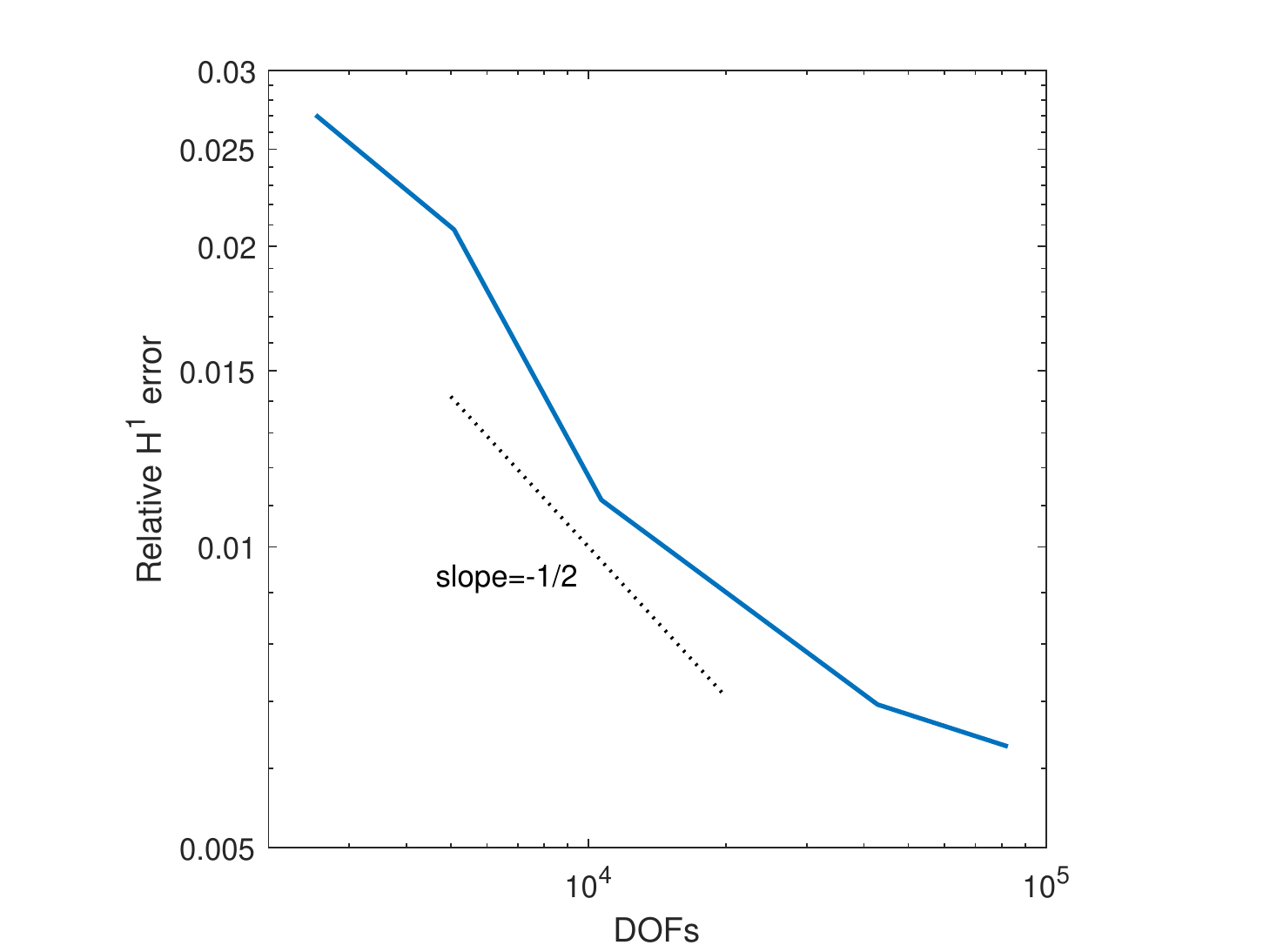}}
  \center{(b).  Tolerance$=10^{-3}$.}
\end{minipage}
\caption{The $H^1$ error of Example~\ref{ex:2dada}} \label{fig:2dada}
\end{figure}

Next we decrease the error tolerance TOL to $10^{-3}$, and we also use the third order reconstruction 
to approximate $\mc{A}$ and there are $4332$ sampling points in total. The error caused by the reconstruction is $e(\text{MOD})=9.48E-04$. The relative $H^1$ error is plot in Fig.~\ref{fig:2dada}b, and we observe that the error decays with optimal rate of convergence $\mc{O}(\text{DOF}^{-1/2})$.

Finally we  compare the accuracy and efficiency between our method and HMM. The offline computation is the same as the previous test, while in the online stage, we solve the homogenized problem with linear element over the same mesh as HMM. Such mesh  is obtained after $8$ iterations with $34394$ triangles from a uniform mesh dividing each of the uniform squares into two sub-triangles; see Fig.~\ref{fig:l2d}b. It is clear that the mesh refinement mostly takes place in the vicinity of the re-entrant corner. It follows from Table~\ref{table:Adaptcomp} that the offline-online method and HMM attain almost the same $H^1$ error, while the running time of the offline-online method is only around $2.46\%$ of HMM. This is easily understood because the number of the cell problems is only  $1.67\%$ of HMM.
\begin{table}[h]
	\caption{Comparison between the offline-online method and HMM.}\label{table:Adaptcomp}
	\vskip .1cm
	\centering
	\begin{tabular}{|c||c|c|c|}
		\hline
		  &$\sharp$ cell problems &Relative $H^1$ error & Time\\
		\hline
		Offline-Online & 576&3.53e-5 & 295.95s\\	
		\hline	
		HMM & 34394 &3.62e-5 &12020.6s \\	
		\hline	
	\end{tabular}
\end{table}

Next we consider a three-dimensional problem in L-shape domain.
\begin{example}\label{ex:3d}
The example is the same with Example~\ref{ex:2dtang} except that
\[
     \a(x)=\frac{(2.5+1.5\sin(2\pi x_1))(2.5+1.5\cos(2\pi x_2))(2.5+1.5\cos(2\pi x_3))}{(2.5+1.5\sin(2\pi x_1/\epsilon))(2.5+1.5\sin(2\pi x_2/\epsilon))(2.5+1.5\sin(2\pi x_3/\epsilon))}1_{3\x 3\x 3},
\]
and $D=(0,1)^3\setminus[0,0.5]^3$. The homogenization problem is the same with~\eqref{eq:2dhomo} with
the effective matrix as
\begin{equation}\label{eq:3dhomocoef}
  \mc{A}(x) = \dfrac15(2.5+1.5\sin(2\pi x_1))(2.5+1.5\cos(2\pi x_2))(2.5+1.5\cos(2\pi x_3))1_{3\x 3\x 3}.
\end{equation}
We let
\[
  u_0(x) = \left((x_1-0.5)^2+(x_2-0.5)^2+(x_3-0.5)^2\right)^{1/10},
\]
and
\[
  g(x) = u_0(x)|_{\pa D},\qquad\qquad f(x) = -\div(\mc{A}(x)\na u_0).
\]
\end{example}

We write the error estimate roughly as
\begin{equation}\label{HMM_LS_Spectral_ExTang1105}
  \nm{\na(u_0-u_h)}{L^2(D)}\leq C\Lr{\text{DOF}^{-1/3}+Q^{-(m+1)}+e^{-\beta N_s}}.
\end{equation}
where DOF is the total degrees of freedom in the online computation, and the factor $e^{-\beta N_s}$ comes from the approximation error caused by using Fourier spectral method~\cite{ShenTangWang:2011} to solve the periodic cell problems~\eqref{eq:percell}, and $\beta$ is a universal constant and $N_s$ is the points we used in each cell.  
The offline refinement strategy reads as
\[
   Q\lesssim\text{TOl}^{-\frac{1}{m+1}}\qquad\text{and\qquad} e^{-\beta N_s}\lesssim\text{TOL}.
\]
We set TOL$=3E-3$ and $m=3$. We firstly divide the domain into $7$ subdomain as in the last example, next we triangulate each subdomain by a uniform mesh with $Q=33$. There are $7623$ sampling points in total; see Fig.~\ref{fig:3d}a. We choose $N_s=25$. 
%
The $H^1$ error is plot in Fig.~\ref{fig:3d}b. We observe that the $H^1$ error decays with an optimal convergence rate $\mc{O}(\text{DOF}^{-1/3})$.
\begin{figure}[h]\centering
\begin{minipage}{0.46\linewidth}
\includegraphics[width=6cm]{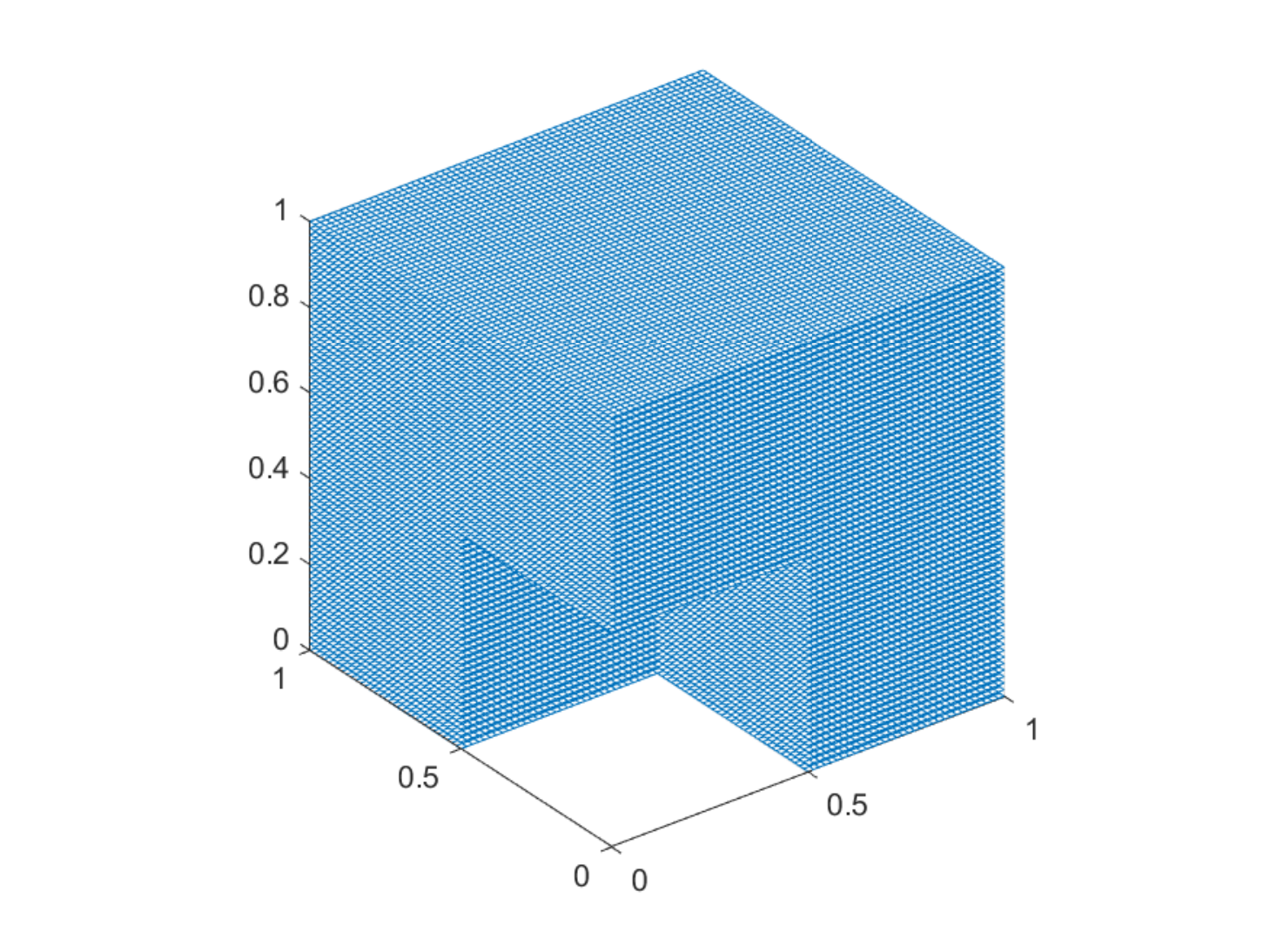}
	\center{(a). Offline sampling mesh $\mc{T}_H$.}
	\end{minipage}
	\hfill
\begin{minipage}{0.48\linewidth}
 \includegraphics[width=6cm]{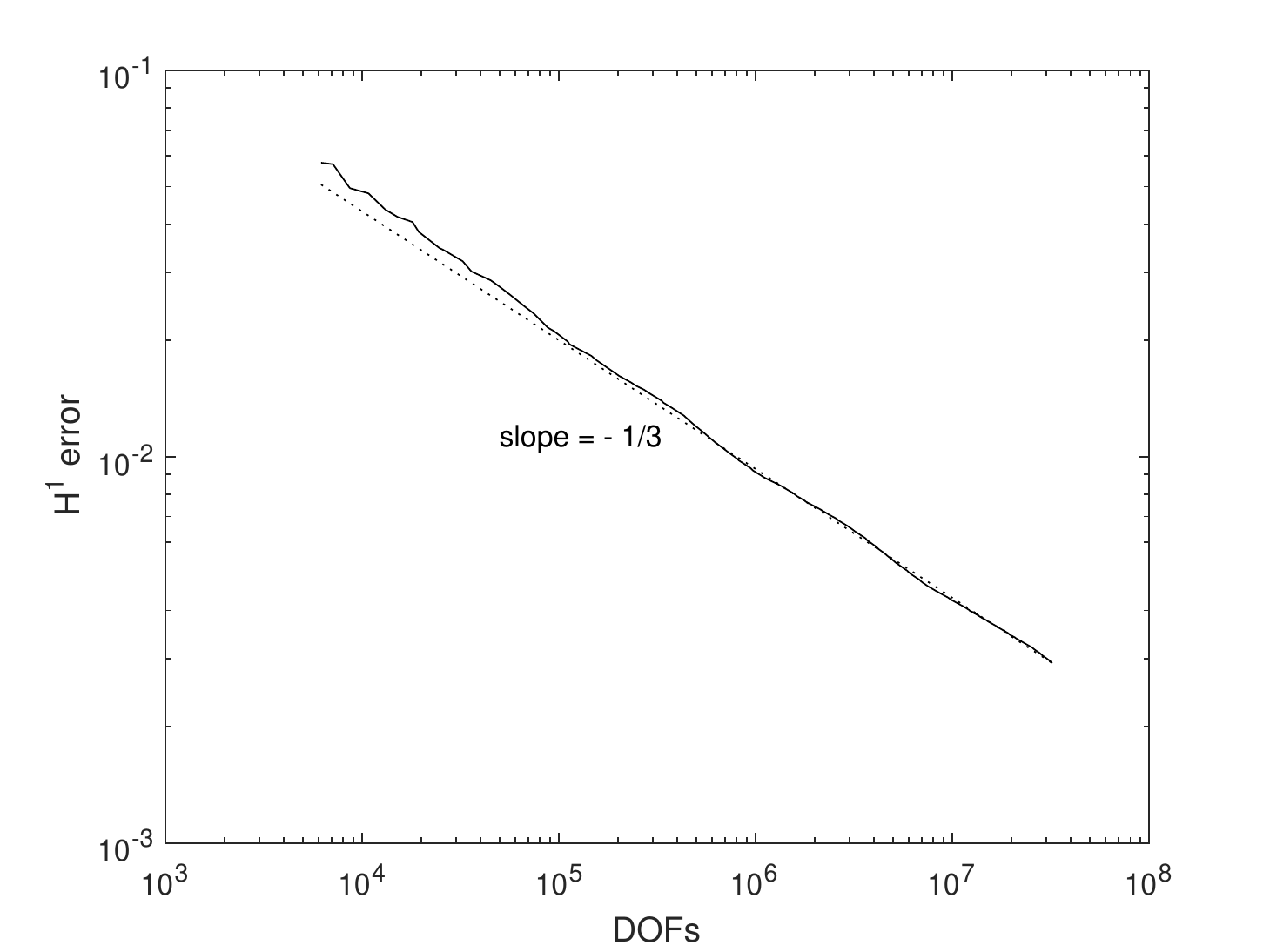}
\center{(b). H$^1$ error with tolerance$=3E-3$.}
\end{minipage}
\caption{Offline triangulation and H$^1$ error for Example~\ref{ex:3d}.}\label{fig:3d}
\end{figure}
\subsection{Example with almost periodic coefficient}
\begin{example}\label{ex:QuasiEx1}
The example is the same with Example~\ref{ex:2dtang} except that
\begin{equation}\label{eq:quasiEx1}
\a(x)=a^0(x/\eps)\,a^1(x)1_{2\x 2},
\end{equation} where
\begin{equation*}
a^0(x)=\left(
\begin{array}{cc}
6+\sin\Lr{2\pi  x_1}^2+\sin\Lr{2\sqrt{2}\pi  x_1}^2&0\\
0&6+\sin\Lr{2\pi  x_2}^2+\sin\Lr{2\sqrt{2}\pi  x_2}^2\\
\end{array}
\right),
\end{equation*}
and
\[
a^1(x)=(2.5+1.5\sin(2\pi x_1))(2.5+1.5\cos(2\pi x_2)).
\]

Such coefficient belongs to Kozlov class~\cite{Kozlov:1978} and the example is adapted from~\cite{Gloria:2016}. The 
coefficients $a^0_{11}(x/\eps)$ and $\a_{11}$ are visualized in Fig.~\ref{fig:quasiCoef} with $\eps=0.1$.
 %
 \begin{figure}[htbp]\centering
	\begin{minipage}{0.48\linewidth}
		\centerline{\includegraphics[width=8cm]{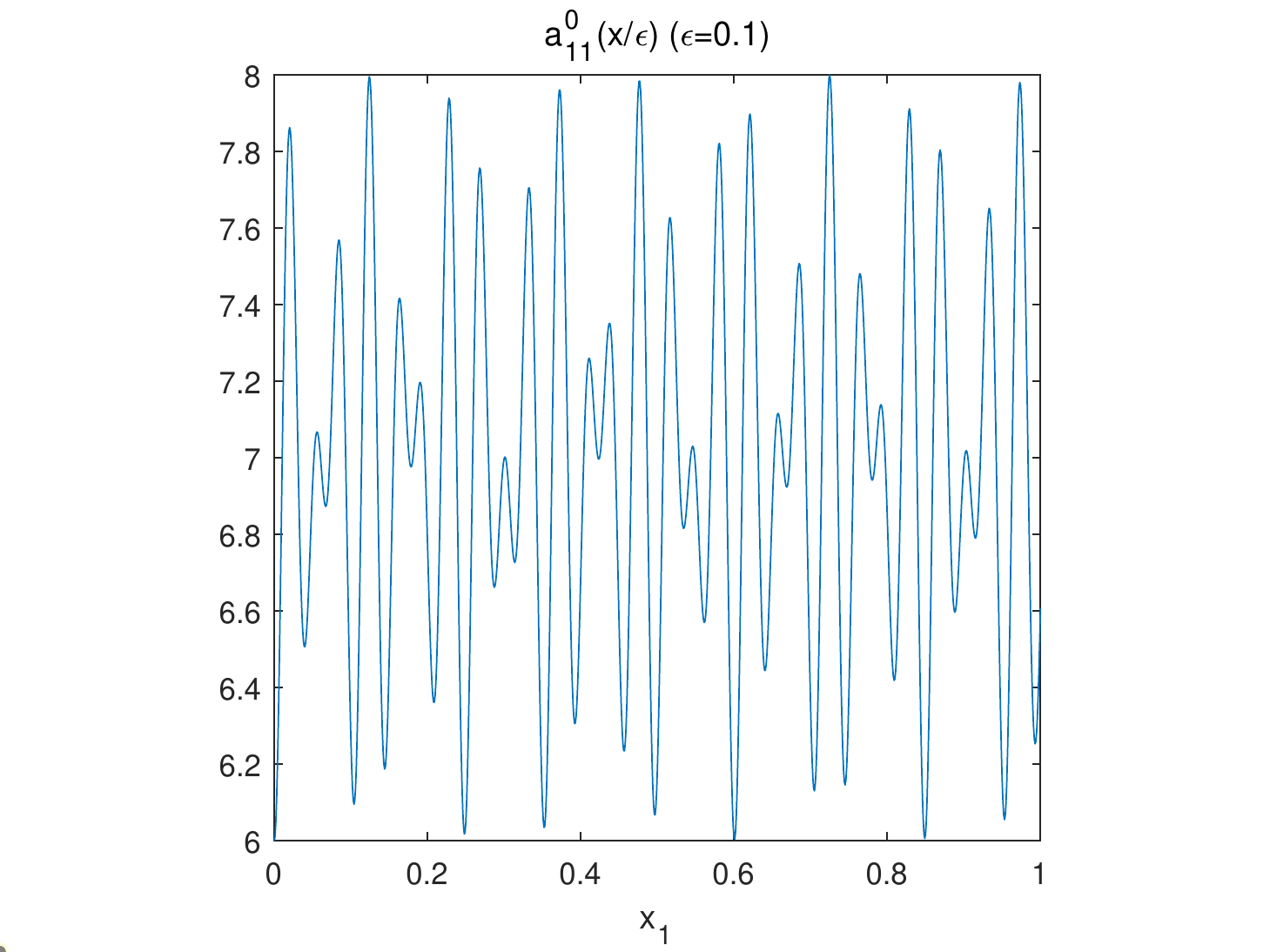}}
		\center{(a). plot of $a^0_{11}(x/\eps)$ .}
	\end{minipage}
	\hfill
	\begin{minipage}{0.48\linewidth}
		\centerline{\includegraphics[width=8cm]{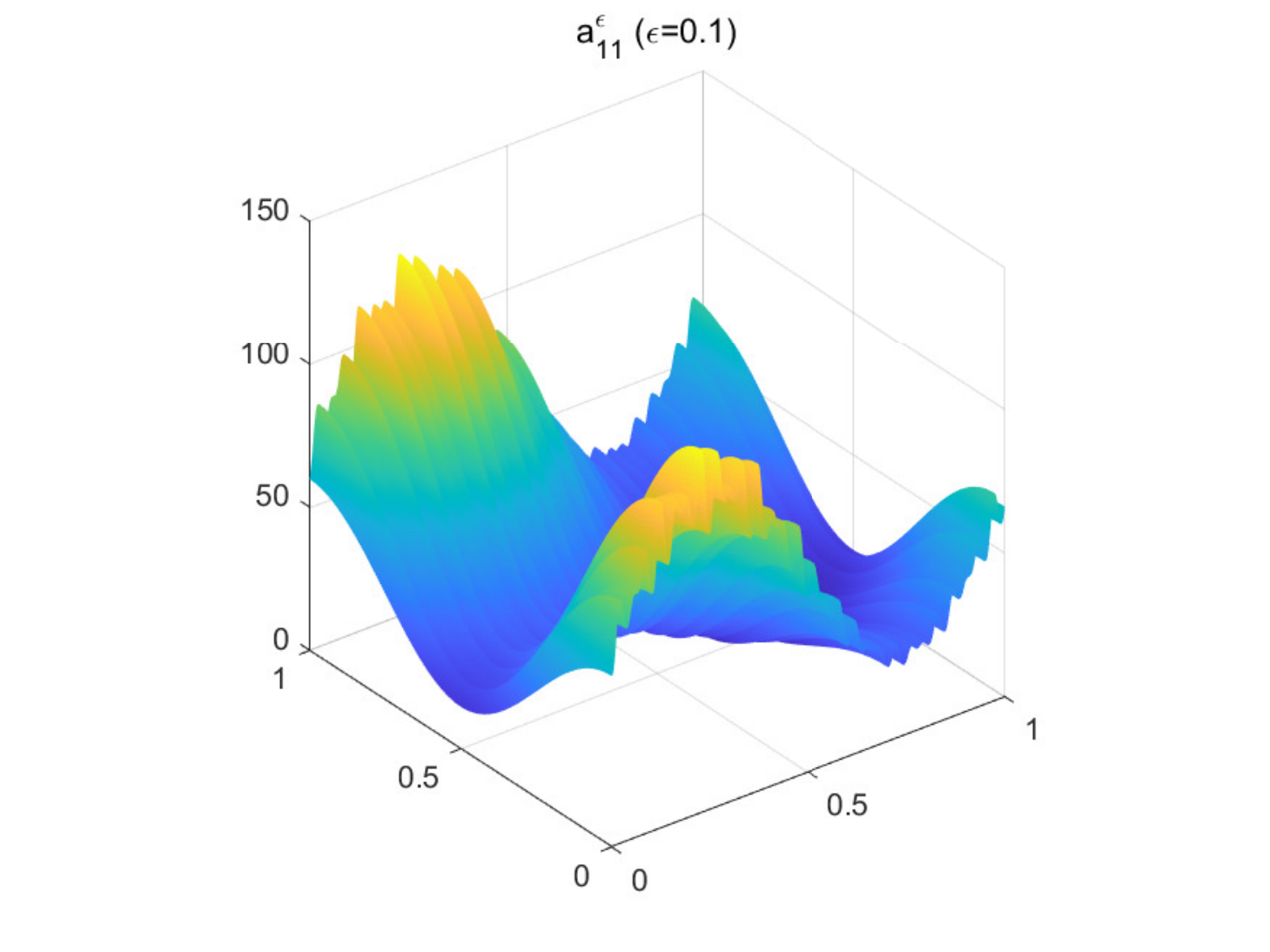}}
		\center{(a). plot of $\a_{11}$.}
	\end{minipage}
	\caption{Plots of the coefficients in Example~\ref{ex:QuasiEx1} ($\eps=0.1$)}\label{fig:quasiCoef}
\end{figure}
The effective matrix is given by
\begin{equation}\label{eq:quasihomo}
\mc{A}_{ij}(x)=\liminf_{R\rightarrow+\infty}\aver{\na(\chi^j+x_j)\cdot a^0(x)\na(\chi^i+x_i)}_{Q_R}* a_1(x),\quad i,j=1,2,
\end{equation}
where $Q_R=(-R,R)^2$ and $\{\chi^i\}_{i=1}^2$ are the solutions of 
\begin{equation}\label{eq:quasiaux}
-\na\cdot(a^0(x)\na \chi^i)=\na \cdot(a^0(x)\na x_i) \quad \text{in\quad} \mb{R}^2.
\end{equation}
As oppose to the locally periodic case, there is no explicit expression of $\mc{A}$.
The naive approach to approximate~\eqref{eq:quasihomo} consists in replacing $\{\chi^i\}_{i=1}^2$ by $\{\chi^i_R\}_{i=1}^2$which are solutions of a truncated cell problem
\begin{equation*}
\left\{\begin{aligned}
-\na\cdot(a^0(x)\na \chi^i_R)=\na \cdot(a^0(x)\na x_i) &\quad \text{in\quad} Q_R,\\
\chi^i_R=0&\quad \text{on\quad}\pa Q_R,
\end{aligned}\right.
\end{equation*}
and $\mc{A}_R$ is defined by
\[
\Lr{\mc{A}_R(x)}_{ij}{:}=\aver{\na(\chi^j_R+x_j)\cdot a^0(x)\na(\chi^i_R+x_i)}_{Q_R}* a_1(x),\quad i,j=1,2.
\]
We take $R=200$ and use $\mb{P}_2$ FEM over a $4000\x 4000$ mesh to solve the above truncated cell problem and obtain 
\[
\mc{A}_{200}=\begin{pmatrix}
		7.00* a^1(x)&0\\
		0&7.00* a^1(x)
		\end{pmatrix}.
\]
\end{example}

In the offline stage, we choose $\delta=10\,\eps$ and use $\mb{P}_2$ FEM over a mesh with size $120\times120$ to solve the cell problems~\eqref{eq:cell}. Next we plot the relative $H^1$ error and the relative $L^2$ error in Fig.~\ref{fig:quasierr}, which clearly shows the higher-order reconstruction is more accurate.
\begin{figure}[h]
	\begin{minipage}{0.48\linewidth}
		\centering
		\includegraphics[width=7cm]{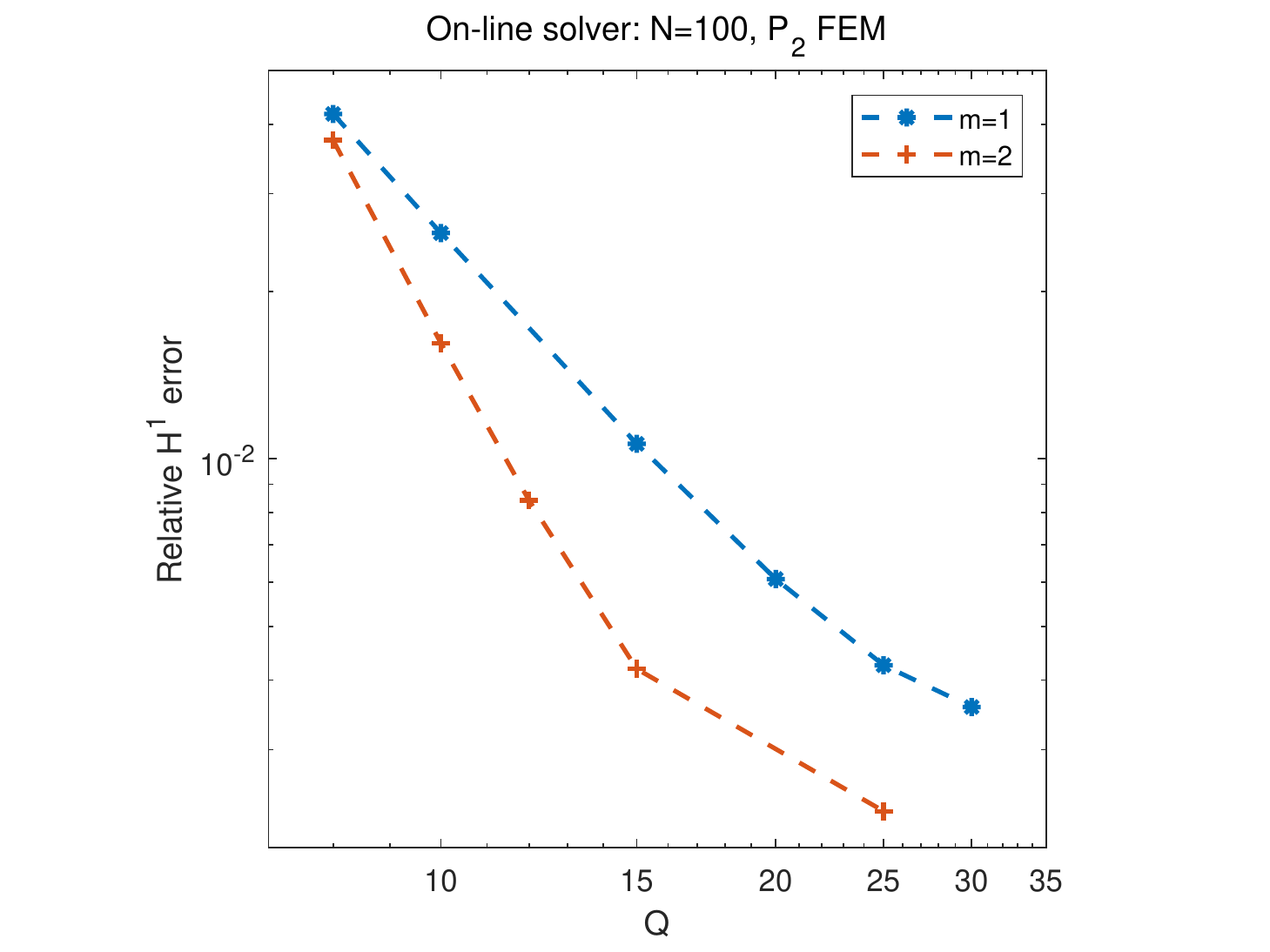}
		\center{(a). The relative $H^1$ error with different reconstruction orders.}
	\end{minipage}
	\hfill
	\begin{minipage}{0.48\linewidth}
		\centering
		\includegraphics[width=7cm]{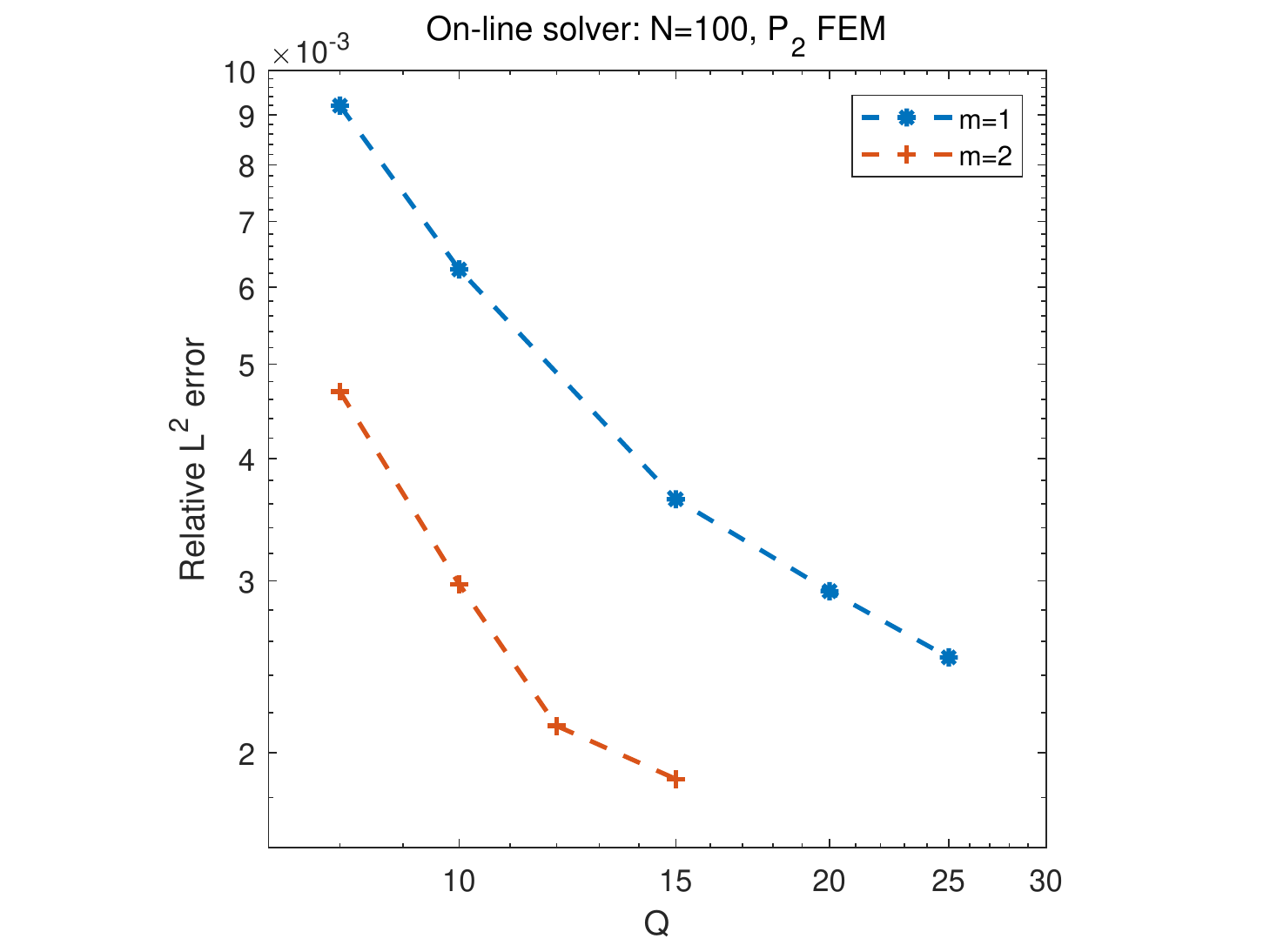}
		\center{(b). The relative $L^2$ error with different reconstruction orders.}
	\end{minipage}
	\caption{The effect of $Q$ and the reconstruction orders in Example~\ref{ex:QuasiEx1}; Online solver is $\mb{P}_2$ FEM and $N=100$.}
	\label{fig:quasierr}
\end{figure}

Finally we compare the number of the cell problems and the running time among reconstructions of different orders in Table~\ref{table:QuasiEx1}.
\begin{table}[h]
	\caption{Comparison of the reconstruction of different orders; online solver is $\mb{P}_2$ FEM and $N=100$.}\label{table:QuasiEx1}
\vskip .1cm
\centering
	\begin{tabular}{|c||c|c|c|c|}
		\hline
		& Q &$\sharp$ cell problems &Relative $H^1$ error & Time\\
		\hline
		$m=1$ & 20 &400&6.08e-3 & 1083.08s\\	
		\hline	
		$m=2$ & 14 &196 &4.95e-3 & 546.24s\\	
		\hline	
	\end{tabular}
\end{table}
The second order reconstruction takes only about $50\%$ of the time used for the first order reconstruction. It is clear that the higher-order reconstruction is more accurate with less cost.
\subsection{Example with random coefficient}
\begin{example}\label{ex:RandEx1}
	The example is the same with Example~\ref{ex:2dtang} except that
	\begin{equation}\label{eq:rand2}
	\a(x)=\Lr{\a_{\text{rand}}+a_0(x)}1_{2\x 2},
	\end{equation}
where $\a_{\text{rand}}$ is a random checker-board, and $a_0(x)=(2.5+1.5\sin(2\pi x_1))(2.5+1.5\cos(2\pi x_2))$. 
$\a_{\text{rand}}$ is constructed by partitioning $D=(0,1)^2$ into uniform square cells of size $\eps$, each of which is randomly designated as $k_1$ or $k_2$ with probability $p_1$ and $p_2=1-p_1$, respectively. We visualize one realization of the random coefficients in Fig.~\ref{fig:randa11} with $\eps=0.02$. Theorem~\ref{thm:main} remains true with a minor modification of $e_1(\text{MOD})$, we refer to~\cite{EMingZhang:2005} for related result.
\begin{figure}[h]
\begin{minipage}{0.48\linewidth}
		\centering
	\includegraphics[width=7cm]{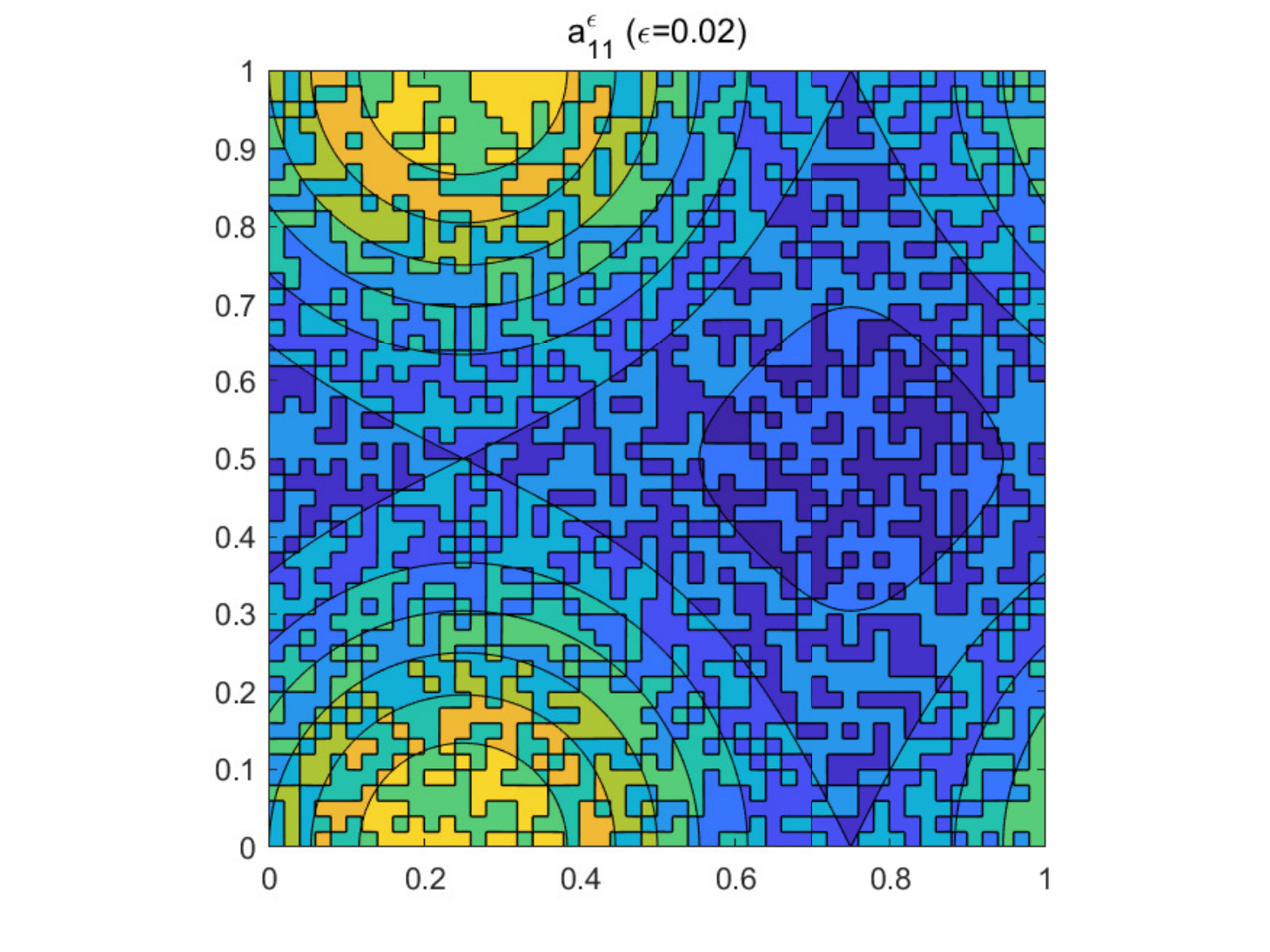}
\end{minipage}
\hfill
\begin{minipage}{0.48\linewidth}
		\centering
\includegraphics[width=7cm]{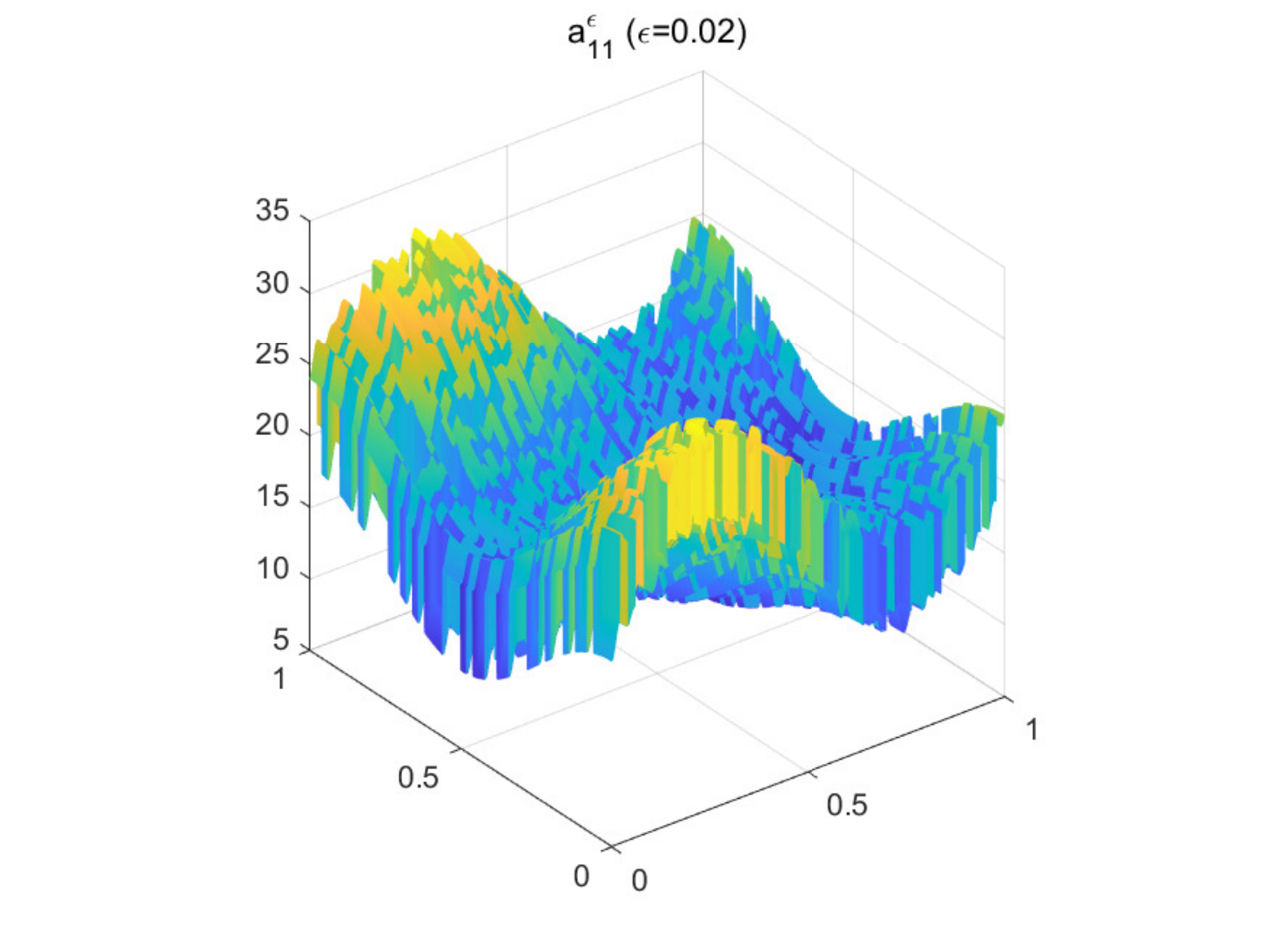}
\end{minipage}
\caption{Plot of one realization of $\a_{11}$ in~\eqref{eq:rand2}. (a) is the contour of $\a_{11}$.}\label{fig:randa11}
\end{figure}
\end{example}

As oppose to the standard random checker-board~\cite{Keller:1964}, there is no explicit formula for the effective matrix of~\eqref{eq:rand2}.  To extract the effective matrix over $D$, we take $\delta=16\,\eps$ as the cell size and use $\mb{P}_2$ FEM over a mesh of size $160\times 160$ to solve the cell problem~\eqref{eq:percell}.  We set the offline mesh size as $Q=20$ and use third order reconstruction to get an approximation effective matrix $\mc{A}$, which will be exploited to obtain  the reference solution in the following test.

To justify this approach, we choose $A_H$ at three representative points in $D$: $A=(1/4,0)$, which is one of the maximum point of $a_0$ over $D$; $B=(3/4,1/2)$, which is one of the minimum point of $a_0$ over $D$,  and $C=(1,1/4)$, which is one of the maximum point of $\na a_0$ over $D$. We let $\delta=L\,\eps$ and use $\mb{P}_2$ FEM over a mesh of size $10L\x 10L$  to solve the cell problems~\eqref{eq:percell} posed over these three points. We denote by 
$\mc{A}_{16}^n$ the approximation effective matrix for $n-$th realization with $L=16$. In addition, we use the empirical average 
\begin{equation}\label{eq:empirical}
\mb{E}(\mc{A}_{16})(x){:}=\dfrac{1}{N}\sum_{n=1}^N\mc{A}_{16}^n(x),
\end{equation} 
as the proxy of the expectation $\mb{E}$, where $N$ is the total number of the realization. We take $N=1000$ in the simulation, and the results are reported in Table~\ref{table:RandEx1}.
\begin{table}[h]
	\caption{The approximating effective matrix on three points with $L=16$.}\label{table:RandEx1}
	\vskip .2cm
	\centering
	\begin{tabular}{|c|c|}
		\hline
		& $\mb{E}(\mc{A}_{16})$\\
		\hline	
		$A=(1/4,0)$ &$\left(
		\begin{array}{cc}
		20.78&-2.58e-4\\
		-2.58e-4&20.78\\
		\end{array}
		\right)$ \\
		\hline	
		$B=(3/4,1/2)$&$\left(
		\begin{array}{cc}
		5.20&-9.47e-4\\
		-9.47e-4&5.20\\
		\end{array}
		\right)$\\
		\hline	
		$C=(1,1/4)$&$\left(
		\begin{array}{cc}
		10.84&-4.86e-4\\
		-4.86e-4&10.84\\
		\end{array}
		\right)$\\	
		\hline		
	\end{tabular}
\end{table}

Secondly, for $x=A,B,C$, we measure the variance as
\[
\sigma^{\text{diag}}_L(x){:}=\sqrt{\mb{E}\Lr{\Lr{\mc{A}_L(x)_{11}-\mb{E}(\mc{A}_{16}(x))_{11}}^2+\Lr{\mc{A}_L(x)_{22}-\mb{E}(\mc{A}_{16}(x))_{22}}^2}},
\]
and
\[
\sigma^{12}_L(x){:}=\sqrt{\mb{E}\Lr{\mc{A}_L(x)^2}_{12}}.
\]
Fig.~\ref{fig:rande1MOD}a and Fig.~\ref{fig:rande1MOD}b suggest that $\sigma^{\text{diag}}_L(x)$  and $\sigma^{12}_L(x)$ decay as $\mc{O}(L^{-1})$, which is consistent with the theoretical predictions~\cite{GloriaOtto:2015}. A systematical numerical tests for the variance may be found in a recent work~\cite{KhOtto:2020}. 
%
\begin{figure}[htbp]
	\begin{minipage}{0.48\linewidth}
		\centering
		\includegraphics[width=7cm]{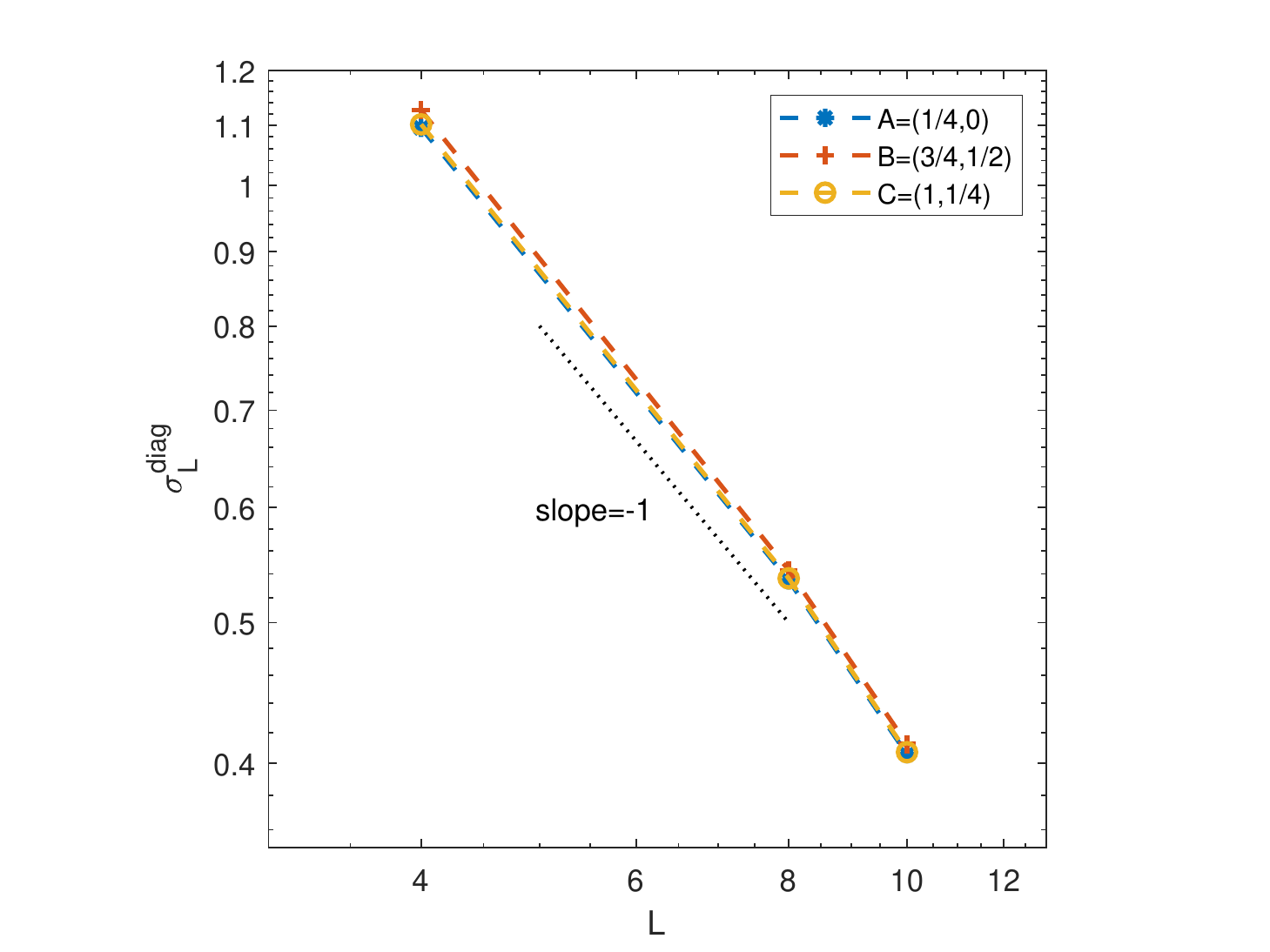}
		\center{(a). Decays of $\sigma^{\text{diag}}_L$ with respect to the cell size $L$.}
	\end{minipage}
	\hfill
	\begin{minipage}{0.48\linewidth}
		\centering
		\includegraphics[width=7cm]{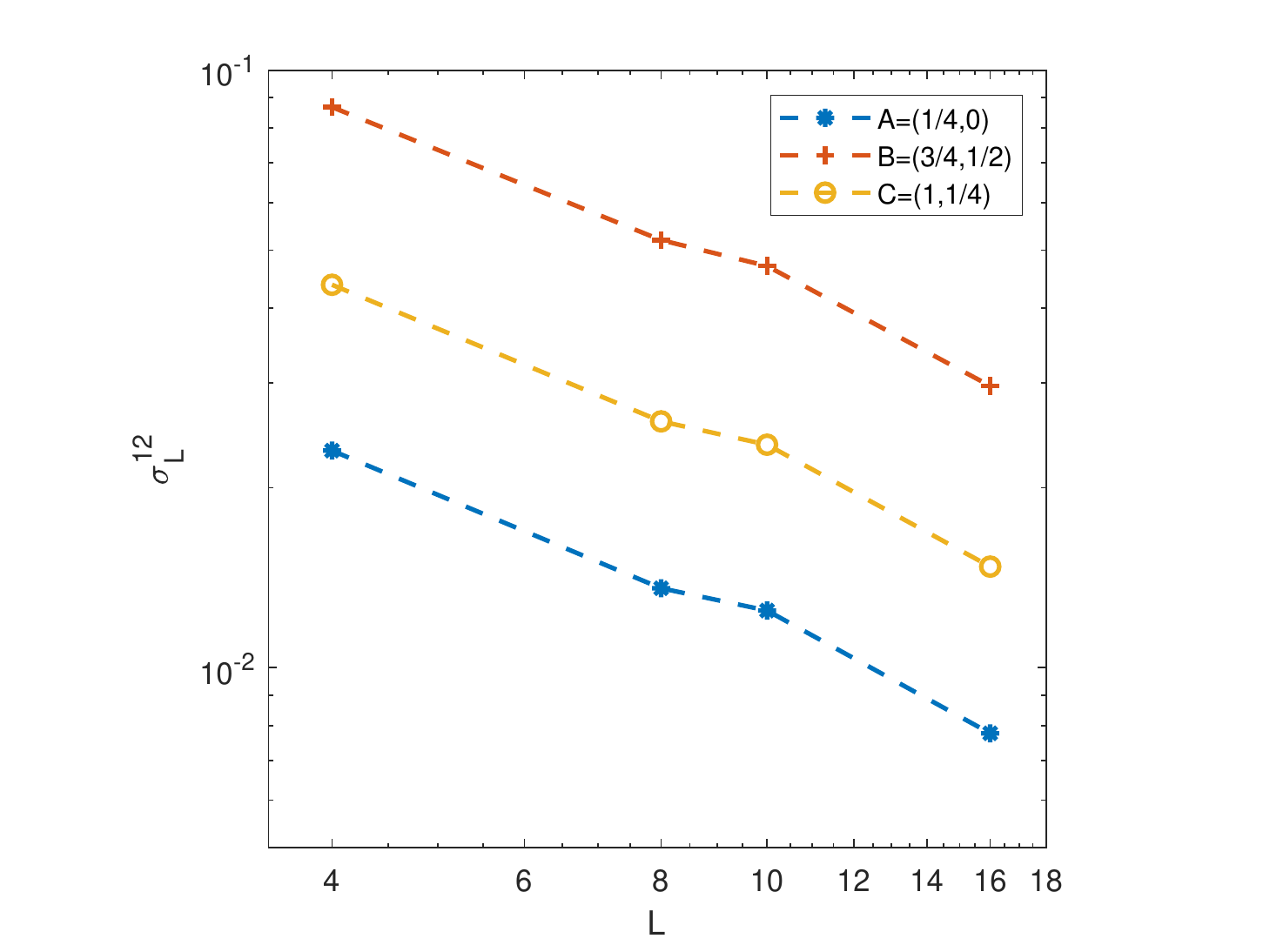}
		\center{(b). Decays of $\sigma^{12}_L$ with respect to the cell size $L$.}
	\end{minipage}
	\caption{The accuracy for the approximating effective matrix.}
	\label{fig:rande1MOD}
\end{figure}
%

Finally we compare the effect of the reconstructions with different orders in the offline computation. We fix the online solver as $\mb{P}_2$ FEM over a mesh with size $100\x 100$. In the offline stage, we choose $\delta=8\,\eps$ and use $\mb{P}_2$ FEM over a mesh with size $80\x 80$ to solve the cell problems~\eqref{eq:cell}.  We take $Q=8,\, 10,\, 16$ and $m=1,\, 2$ in the tests, and compute the ensemble average of  the relative $H^1$ error and $L^2$ error
\[
   \dfrac{\mb{E}\Lr{\nm{\na(u_0-u_h)}{L^2(D)}}}{\nm{\na u_0}{L^2(D)}}\quad \text{and} \quad \dfrac{\mb{E}\Lr{\nm{u_0-u_h}{L^2(D)}}}{\nm{u_0}{L^2(D)}}.
\]
The expectation $\mb{E}$ is replaced by the empirical average as that in~\eqref{eq:empirical} with $N=1000$ realizations. The visualization in Fig.~\ref{fig:randAdderr}a and Fig.~\ref{fig:randAdderr}b clearly shows that the second order reconstruction is more accurate than the first order reconstruction.
\begin{figure}[h]
	\begin{minipage}{0.48\linewidth}
		\centering
		\includegraphics[width=7cm]{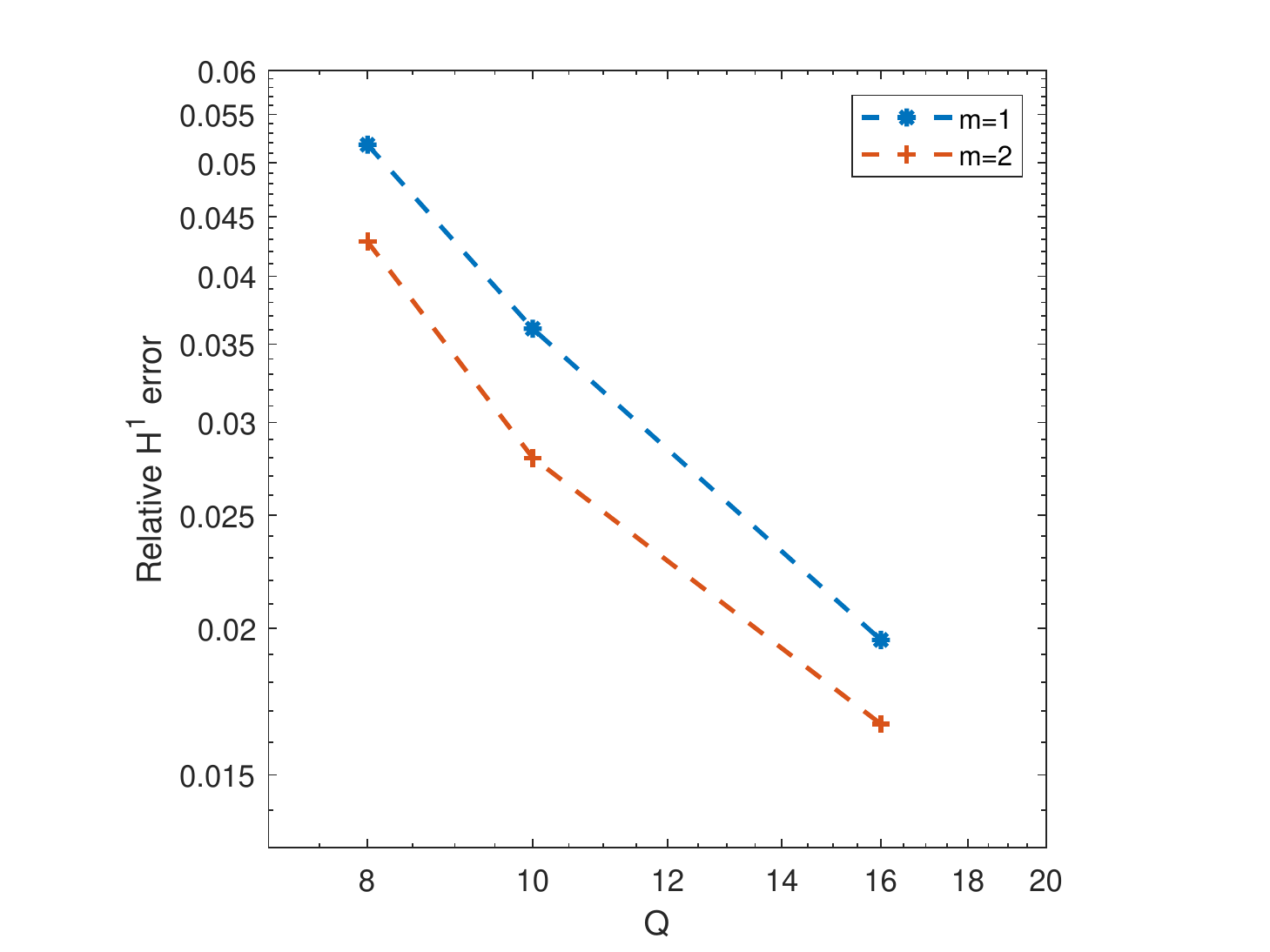}
		\centerline{(a). Ensemble average of the $H^1$ error.}
	\end{minipage}
\hfill
\begin{minipage}{0.48\linewidth}
		\centering
		\includegraphics[width=7cm]{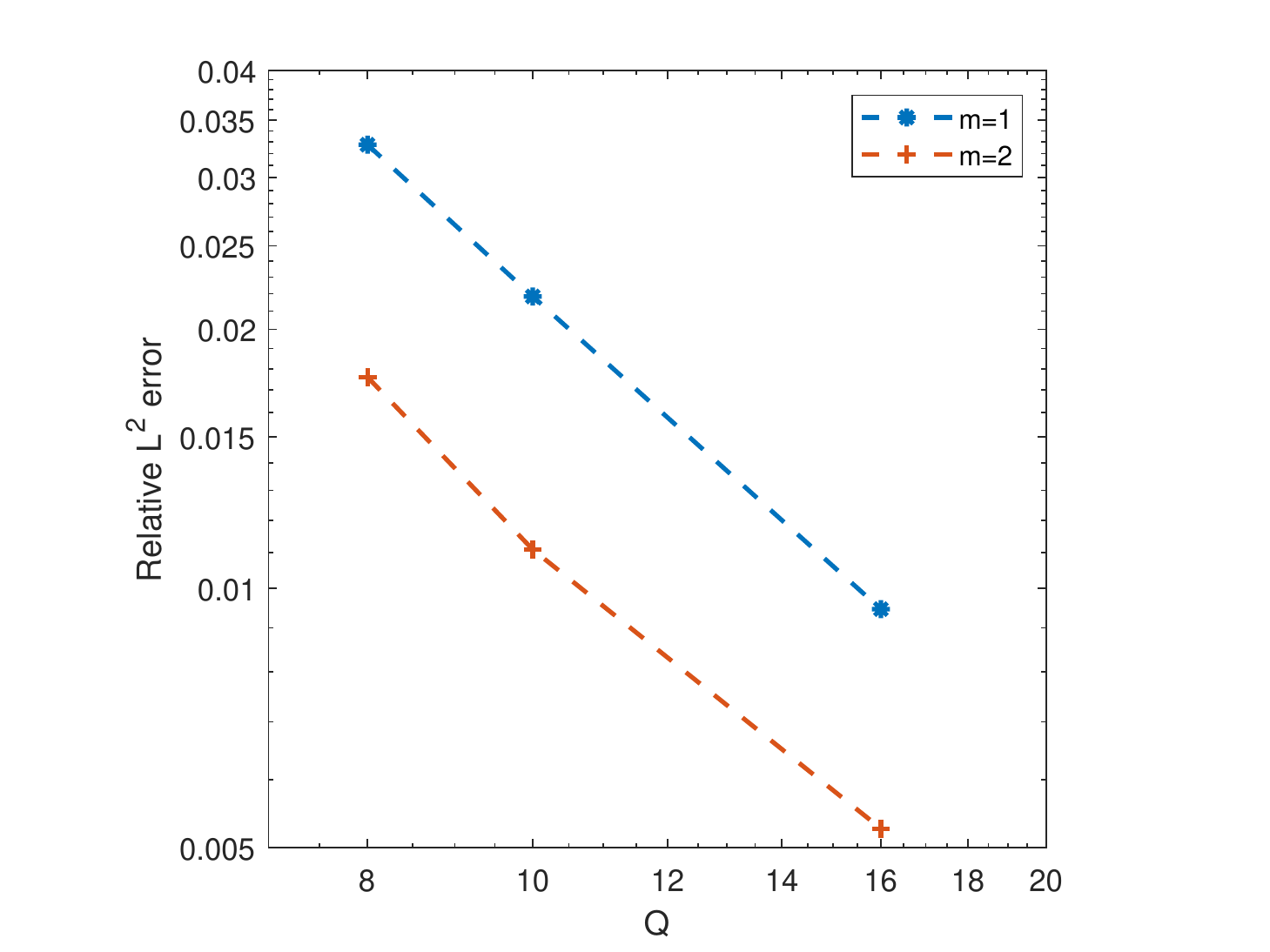}
		\center{(b). Ensemble average of the $L^2$ error.}
	\end{minipage}
	\caption{The error between numerical solution and homogenized solution in Example~\ref{ex:RandEx1}.}
	\label{fig:randAdderr}
\end{figure}
\section{Conclusion}
 We have proposed a new online-offline method to solve the multiscale elliptic problems. Both theoretical and numerical results show that the method significantly reduces the cost while retains the optimal rate of convergence. Moreover, the strategy is problem independent, and it can be extended to time-dependent problems~\cite{MingZhang:2007, Engquist:2012}. The implementation of the present method is mainly based on the a priori error estimate. Adaptive algorithms based on the a priori error estimate should be developed for automatic tuning of the parameters so that the method is more efficient. We shall leave all these issues in the future work. 

\end{document}